\documentclass[11pt,a4paper]{article}

\usepackage[english]{babel}
\usepackage{ragged2e}
\usepackage[utf8]{inputenc}
\usepackage[T1]{fontenc}
\usepackage{amsmath}
\usepackage{amsfonts}
\usepackage{amssymb}
\usepackage{graphicx}
\usepackage{bbm}
\usepackage{amsmath,amsfonts,amsthm,amssymb,amscd, appendix, verbatim}
\usepackage{xpatch}
\usepackage{parskip}
\usepackage{setspace, color}
\usepackage{array}
\usepackage{cite}
\usepackage{hyperref}

\newtheorem{theorem}{Theorem}[section]
\newtheorem{corollary}[theorem]{Corollary}

\newtheorem{lemma}[theorem]{Lemma}
\newtheorem{proposition}[theorem]{Proposition}

\newtheorem{remark}[theorem]{Remark}

\numberwithin{equation}{section}

\newcommand{\dem}{\noindent \textit{Proof.  }}
\newcommand{\1}{\mathbbm{1}}

\newcommand{\Z}{\mathbb{Z}}
\newcommand{\F}{\mathcal{F}}
\newcommand{\Sc}{\mathcal{S}}

\newcommand{\M}{\mathcal{M}}

\newcommand{\R}{\mathbb{R}}

\newcommand{\supp}{\operatorname{supp}}

\title{Sharp well-posedness for the $k$-dispersion generalized Benjamin-Ono equations: Short and long time results}

\date{}
\author{
Luccas Campos \,
Felipe Linares \,
Thyago S. R. Santos \,
}
\begin{document}
\maketitle
\begin{abstract}

\noindent We consider the $k$-dispersion generalized Benjamin-Ono ($k$-DGBO) equations. For nonlinearities with power $k \geq 4$, we establish local and global well-posedness results for the associated initial value problem (IVP) in both the critical and subcritical regimes, addressing sharp regularity in homogeneous and inhomogeneous Sobolev spaces. Additionally, our method enables the formulation of a scattering criterion and a scattering theory for small data. We also investigate the case $k = 3$ via frequency-restricted estimates, obtaining local well-posedness results for the IVP associated with the $3$-DGBO equation and generalizing the existing results in the literature for the whole subcritical range. For higher dispersion, these local results can be extended globally even for rough data, particularly for initial data in Sobolev spaces with negative indices. As a byproduct, we derive new nonlinear smoothing estimates.
\end{abstract}

\section{Introduction}
We study the Cauchy problem related to the family of $k$-dispersion generalized Benjamin-Ono ($k$-DGBO) equations:
\begin{equation}\label{DGBO}
\left\{
	\begin{array}{l}
	u_t + D_x^\alpha u_x + \mu \, u^ku_x= 0, \quad (t,x)  \in \R \times \R,\\
		u(0,x)=u_0(x)
	\end{array}
	\right.
\end{equation}
 where $u = u(t,x)$ is real-valued, $\alpha \in [1,2]$, $\mu \in \{\pm 1\}$  and $k \in \Z^+$. Here, ${D^\alpha_x}$ represents the fractional Laplacian, defined through the Fourier transform as follows:
 $$
 {D^\alpha_x f}:= \left(|\xi|^\alpha \widehat{f}(\xi)\right)^\vee={(\mathcal{H}\partial_x)^{\alpha}f},
 $$
 where  $\mathcal{H}$ represents the Hilbert transform

\begin{equation}\label{hilberttranform}
(\mathcal{H} f) (x):= \text{\textbf{p.v.}}\,\frac{1}{\pi}\int_{\R} \frac{f(y)}{x-y}\, dy := \lim_{\varepsilon \searrow 0}\,\frac{1}{\pi}\int_{|x-y|> \varepsilon} \frac{f(y)}{x-y}\, dy.
\end{equation}


 
In the particular case $k=1$, the class of equations in \eqref{DGBO} arise as mathematical models for the unidirectional propagation of weakly nonlinear long waves (\hspace{0.02cm}\cite{MR273890}, \cite{MR1422759}). In particular, we notice that for $\alpha=2$, the equation \eqref{DGBO} reduces to the well-known Korteweg-de Vries (KdV) equation
$$
u_t + u_{xxx} + \mu \, uu_x =0
$$
and for $\alpha=1$, the equation becomes the Benjamin-Ono (BO) equation
$$
u_t + \mathcal{H}u_{xx} + \mu \, uu_x =0.
$$
The complete integrability of the equations \eqref{DGBO} remains uncertain in general. Nevertheless, their solutions satisfy, for all $t>0$, the conservation of mass
\begin{equation*}
M(u(t)):=\int_\mathbb{R} u^2(x,t)\, dx= M(u_0)
\end{equation*}
and the conservation of energy
\begin{equation*}
E(u(t)):=\dfrac{1}{2}\int_{\mathbb{R}}|D_x^{\frac{\alpha}{2}} u(x,t)|^2\;dx
-\dfrac{\mu}{k+2}\int_{\mathbb{R}} u^{k+2}(x,t)\;dx= E(u_0).
\end{equation*}
Additionally, the equations in \eqref{DGBO} enjoy a scaling invariance, \textit{i.e.} if $u$ is solution of \eqref{DGBO}, then,  for all $\lambda>0$,
 $$
 u_{\lambda}(t,x):= \frac1{\lambda^{\frac{\alpha}{k}}}u\left(\frac{t}{\lambda^{\alpha+1}}, \frac{x}{\lambda}\right)
 $$
 also defines a solution of \eqref{DGBO} with initial data $u^\lambda_0 (x) := \lambda^{-\frac{\alpha}{k}}u_0 (\lambda^{-1} x)$. Furthermore
$$
\| u_{\lambda}(t,\cdot)\|_{\dot{H}^{s_k}_x} = \| u(t,\cdot)\|_{\dot{H}^{s_k}_x},
$$
where $\dot{H}^s_x(\R)$ denotes the homogeneous Sobolev space (see \eqref{sobolev homogeneo} below) and
 \begin{equation}\label{scale}
 s_k:=\frac{1}{2} - \frac{\alpha}{k}.
 \end{equation}
The real number $s_k$ is referred to as the \textit{Sobolev critical index}, as it represents the minimal regularity for which one can expect local well-posedness for the equations \eqref{DGBO}.

Concerning local well-posedness for the IVP \eqref{DGBO}, $k\ge2$, Kenig, Ponce and Vega showed in \cite{MR1086966}  that the IVP \eqref{DGBO} is locally well-posed in $H_x^s(\R)$,  for $s\ge(9-3\alpha)/4$ and globally well-posed
in $H_x^{\frac{\alpha}{2}}(\R)$ for $\alpha\ge9/5$. In  \cite{MR2859931}, Kenig, Martel and Robiano considered the family of dispersion generalized Benjamin-Ono equations,
\begin{equation}\label{cri-dgbo}
u_t-D_x^{\alpha}u_x+|u|^{2\alpha}u_x=0,
\end{equation}
which are critical with respect to both the $L_x^2$-norm and the global existence of solutions. Observe that the equations \eqref{cri-dgbo} are the
interpolated ones between the modified BO equation ($\alpha=1$) and the critical KdV equation ($\alpha=2$).  In \cite{MR2859931}
it was proved, among other results, the local well-posedness in the energy space $H_x^{\frac{\alpha}{2}}(\R)$ for $\alpha\in (1,2)$.
They also studied the blow-up problem in the same fashion as in \cite{MR1896235} and \cite{MR1824989}. More specifically, they established the finite time
blow-up of solutions with negative energy in the energy space when $\alpha$ is close to 2.

Several global well-posed aspects were studied by Farah, Linares and Pastor \cite{MR3200754} for solutions of \eqref{DGBO}. In particular, they showed
sharp conditions on the initial data to establish the global well-posedness theory in the energy spaces $H_x^{\frac{\alpha}{2}}(\R)$. We refer to \cite{MR3200754} for a more complete account of local/global issues regarding the Cauchy problem \eqref{DGBO}.

\medskip

The main goal of this work is to establish well-posedness results in Sobolev spaces $ \dot{H}^s_x(\mathbb{R}) $ and $ H^s_x(\mathbb{R}) $ for $ s \geq s_k $, addressing both the critical case ($ s = s_k $) and subcritical case ($ s > s_k $), thereby achieving the optimal regularity index for well-posedness. Additionally, we examine some qualitative properties of the solutions to the equations \eqref{DGBO}, including propagation of regularity and long-time behavior.

First, we consider the case where $k \geq 4$ in \eqref{DGBO}. Let $I \subset \R$ be an interval. For $ s\geq 0$ and $\alpha > 1$, we define the resolution space being
\begin{equation}
    \dot{X}^{s}(I) :=  
    C_t \dot H^s_x (I \times \mathbb R) \cap \dot W^{s+\frac{1}{2},\frac{4\alpha+2} {\alpha-1}}_xL^{\frac{4\alpha+2}{3}}_t (I \times \mathbb R),
\end{equation}
endowed with the canonical norm. We also fix the following indices for $(\alpha,s) \neq (1,s_k)$ 
 $$
 {p_s}= {p}(s,\alpha,k):= \frac{(2\alpha+1)k}{\alpha+2},
 $$
 and 
  $$
 {q_s}= {q}(s,\alpha,k):= \frac{(2\alpha+1)(\alpha+1)k}{2(\alpha^2-1) - (s-s_k)(2\alpha+1)k},
 $$
 which allows us to define the auxiliary spaces
\begin{equation*}
 S^{s}(I) = L^{ p_s}_xL^{ q_s}_t(I \times \mathbb R).
\end{equation*}

\begin{remark}
When $I= [0,T]$, we denote
$$
\dot{X}^{s}([0,T]):= \dot{X}^{s}_T \quad \text{and}\quad  S^{s}([0,T]):=  S^{s}_T.
$$
\end{remark}
 \begin{remark}\label{Remark sobre Ss}
The auxiliary spaces $S^s_T$ play an important role in the asymptotic behavior. In fact, the finiteness of this norm allows the solution to be extended to larger time intervals, and its uniform boundedness for all $T > 0$, ensures global well-posedness and scattering. These topics will be further explored throughout this work.
 \end{remark}
 
We denote by $\{V_{\alpha}(t)\}_{t \in \R}$ the free evolution group associated with the linear part of the equation \eqref{DGBO} (see \eqref{free evolution group} below). The first result of this work is as follows:

\medskip

 \begin{theorem}[Critical local well-posedness for $k\geq 4$]\label{Teorema existencia critica} Let $\alpha \in (1,2]$ and $k \geq 4$. Consider $s_k=\frac{k-2\alpha}{2k}$ and $u_0 \in \dot H_x^{s_k}(\R)$. Then, there exist $\varepsilon_0=\varepsilon_0\left(\|u_0\|_{\dot H_x^{s_k}}\right)>0,$ and a unique solution $u$ of \eqref{DGBO} with initial data $u_0$ such that
 $$
 u \in  \dot{X}^{s_k}_T \cap S^{s_k}_T,
 $$
for $T=T(u_0)>0$ chosen to satisfy
\begin{equation*}
    \|V_\alpha(t)u_0\|_{S^{s_k}_T} \leq \varepsilon_0.
\end{equation*}
Moreover, the data-solution map $u_0 \mapsto u \in \dot{X}^{s_k}_T \cap S^{s_k}_T $ is locally uniformly continuous.
\end{theorem}
The next result ensures that the regularity of the initial data is preserved for the solution obtained in Theorem \ref{Teorema existencia critica}.
\begin{proposition}[Propagation of regularity]\label{teo:propagation_regularity} Let $\alpha \in (1,2]$, $k \geq 4$ and $u_0 \in \dot H_x^ s \cap \dot H_x^ {s_k}$ for $s \geq 0$. Then the solution $u$  given by Theorem \ref{Teorema existencia critica} belongs to $ \dot X^s_T \cap S^s_T$.
\end{proposition}

As noted in Remark \ref{Remark sobre Ss}, controlling the $S^s(I)$-norms helps us understand the long-time behavior of solutions to \eqref{DGBO}. More specifically, we focus on the scattering of small-amplitude solutions. In broad terms, the problem can be described as follows: given initial data with a finite norm in $S^{s_k}([0,+\infty))$, the corresponding solution $u(t)$ approaches a linear evolution as time progresses towards infinity. More precisely, we present the following result:

\begin{proposition}[\textit{Scattering criterion}]\label{prop:scattering_criterion}
Let $u_0 \in \dot H^{s_k}_x(\R)$ and suppose that the corresponding solution $u$ to \eqref{DGBO} is defined for all $t \geq 0$. If
\begin{equation*}
    \|u\|_{S^{s_k}({[0,+\infty)})} < + \infty,
\end{equation*}
then $u$ scatters forward in $\dot H^{s_k}_x(\R)$. In other words, there exists $u_+ \in \dot H_x^{s_k}(\R)$ such that
\begin{equation*}
    \lim_{t \to +\infty}\|u(t) - V_{\alpha}(t) u_+\|_{\dot H_x^{s_k}} = 0.
\end{equation*}

    Moreover, if $u_0 \in \dot H^{s}_x(\R)$ for some $s\geq-\frac{1}{2}$, then $u$ also scatters forward in $\dot H^{s}_x(\R)$. An analogous result holds on the time interval $(-\infty, 0]$ as well.
\end{proposition}

As a result, we establish the following global-in-time theory for small initial data.

\begin{corollary}[Critical global theory for small data  ]\label{cor:critical_small_data} For $\alpha \in (1,2]$ and $k \geq 4$ there exists a number $\delta_{sd}>0$ such that, if $u_0 \in \dot{H}_x^{s_k}(\R)$ satisfies
\begin{equation}\label{time_well_posedness}
    \|u_0\|_{\dot H_x^{s_k}} < \delta_{sd},
\end{equation}
then the unique local solution obtained in Theorem \ref{Teorema existencia critica}  can be extended globally. Moreover, $u$ scatters forward in $\dot{H}_x^{s_k}(\R)$.
\end{corollary}

Our next task will be to focus on the reciprocal problem of Proposition \ref{prop:scattering_criterion}, specifically constructing solutions to \eqref{DGBO} with a prescribed asymptotic behavior.

\begin{theorem}\label{wave operator}
Let $\alpha \in (1,2]$ and $k \geq 4$. For any $v_0 \in \dot{H}^{s_k}_x(\R)$ there exist $T_0=T_0(\|v_0\|_{\dot H^s_x}) \gg 1$ and $u \in C([T_0,\infty):\dot{H}^{s_k}_x(\R))$ a solution of IVP \eqref{DGBO} satisfying 
\begin{equation*}
\lim_{t \rightarrow +\infty}\| u(t) - V_\alpha(t)v_0 \|_{\dot{H}^{s_k}_x} = 0.
\end{equation*}
\end{theorem}

 From now on, we will refer to Theorem \ref{wave operator} as the \textit{construction of a wave operator} for the equation in \eqref{DGBO}. In the following, we will consider the local Cauchy theory in the subcritical case. Our first result in this direction is as follows:

 \begin{theorem}[Subcritical well-posedness]\label{teo:subcritical_lwp}  Let $\alpha \in (1,2]$ and $k \geq 4$.
 
 \begin{itemize}
 \item [1.] (Homogeneous Case) Let $s_k < s < s_k + \frac{2(\alpha^ 2-1)}{(2\alpha+1)k}$ and $u_0 \in \dot H_x^ s(\R)$. Then, there exist $T(\|u_0\|_{\dot H^s_x})>0$ and a unique solution $u$ to \eqref{DGBO}  such that $u \in  \dot X^s_T \cap S^s_T$.  

 \item[2.] (Inhomogeneous Case) Let $s>s_k$ and $u_0 \in H_x^ s(\R)$. Then, there exist $T(\|u_0\|_{\dot H^s_x})>0$ and a unique solution $u$ to \eqref{DGBO}  such that $u \in  \dot X^s_T \cap S^s_T$. 
 \end{itemize}
 In both cases, the data-solution map $u_0 \mapsto u \in \dot{X}^{s}_T \cap S^{s}_T $ is locally uniformly continuous.
\end{theorem}
The Theorem \ref{teo:subcritical_lwp} can also be used to understand the asymptotic behavior of solutions, mainly. 

 \begin{corollary}[Blow-up alternative]\label{cor:subcritical_blowup_rate} Let $\alpha \in (1,2]$, $k \geq 4$ and $s_k  < s < s_k + \frac{2(\alpha^ 2-1)}{(2\alpha+1)k}$.
 
 \begin{enumerate}
 \item[1.] (Subcritical case) If $u_0 \in \dot H_x^s(\R) \cap \dot H_x^{s_k}(\R)$. Let $T^*> 0$ be the maximal time of existence of the corresponding solution $ u \in \dot X^s_{T^*}\cap S^s_{T^*}$. Then, if $T^*<\infty$ we have
\begin{equation*}\label{infinite_subcritical_scattering_norm}
    \|u\|_{S^s_{T^*}} = +\infty.
\end{equation*}
Moreover, we have the following blow-up rate
\begin{equation*}
    \|u(t)\|_{\dot H_x^s} \gtrsim \frac{1}{(T^* -t)^{\frac{s-s_k}{\alpha+1}}}.
\end{equation*}
\item[2.] (Critical case) If $u_0 \in \dot H^{s_k}_x(\R)$ and the maximal time of existence $T^*>0$ of the corresponding solution 
$$
u \in \dot X^{s_k}_{T^*} \cap S^{s_k}_{T^*}
$$
is finite, then
\begin{equation*}
    \|u\|_{ S^{s_k}_{T^*}} = +\infty.
\end{equation*}
\end{enumerate}

Analogous results holds for any $s \geq s_k$ in the inhomogeneous case.
\end{corollary}

\medskip

\begin{remark}
It is important to emphasize that the results presented so far can be analogously extended to the complex-valued case
  \begin{equation*}\label{complex-DGBO}
	\left\{
	\begin{array}{l}
	u_t + D_x^\alpha u_x + |u|^ku_x= 0, \quad (t,x)  \in \R \times \R,\\
		u(0,x)=u_0(x),
	\end{array}
	\right.
\end{equation*}
with $u=u(t,x) \in \mathbb{C}$ and $\alpha \in (1,2]$.
\end{remark}

Finally, we focus on the case $k=3$. In this situation, we obtain the following local well-posedness result.


\begin{theorem}[Local well-posedness for $k=3$]\label{Teorema do XSB da introdução}
Let $\alpha \in (1,2]$, $k=3$, and $s>s_3 = \frac{1}{2}-\frac{\alpha}{3}$. Given $u_0 \in H^s_x(\R)$, there exist $T=T\left(\|u_0\|_{H^s_x}\right)>0$ and a unique solution $u \in C_TH^s_x$ to the problem \eqref{DGBO}.  Furthermore, the data-solution map $ u_0 \mapsto u \in  C_TH^s_x$ is locally uniformly continuous.
\end{theorem}
\begin{remark}
If we set $\alpha=2$ and $k=3$ in \eqref{DGBO}, Theorem \ref{Teorema do XSB da introdução} recovers the well-posedness result obtained by Gr\"unrock in \cite{MR2174975}.
\end{remark}

\begin{remark}
Farah, Linares, and Pastor established local well-posedness for \eqref{DGBO} in $H_x^{\frac{\alpha}{2}}(\mathbb{R})$ for $\alpha \in (1,2)$ and $k > 2\alpha$ (see \cite{MR3200754}, Theorem 2.3). The Theorems \ref{teo:subcritical_lwp} and \ref{Teorema do XSB da introdução}, however, not only significantly improve upon this local well-posedness result for $s\geq s_k$ but also extend the analysis to include the case $\alpha = 2$.
\end{remark}

In \cite{restrictionmethod}, Correia, Oliveira, and Silva, propose the following conjecture:

\textbf{Conjecture:}\textit{(\hspace{-0.13cm} \cite{restrictionmethod}, Conjecture 1). Given a nonlinear dispersive equation with dispersion of order $l$ and a nonlinearity of order $k \geq  3$ with a loss of $m$ derivatives, let $s_0$ be such that, for $s>s_0$, the $H^s_x$-flow exists and is analytic. Then the flow exhibits a nonlinear smoothing effect of order $\varepsilon  < \min\{ (k-1)(s-s_0), l-m-1\}$.}

Thus, the following result lends support to this conjecture in this specific case.

\begin{theorem}[Nonlinear smoothing effect for $k=3$]\label{efeito regularizante nao linear}
Let $\alpha \in (1,2]$, $k=3$, and $u_0 \in H^s_x(\mathbb{R}) $ with $ s > s_3$. The solution $ u \in C_T H^s_x $ of  \eqref{DGBO}, provided by Theorem \ref{Teorema do XSB da introdução}, satisfies
$$
u(t) - V_\alpha(t)u_0 =\mu \int_0^t V_\alpha(t-\tau)(u^3\partial_xu)(\tau)\, d\tau \in H^{s+\varepsilon}_x(\R),
$$
for all $\varepsilon \in \left[0, \min \left\{3s+\alpha - \frac{3}{2}, s+1/2, \alpha-1\right\}\right)$.
\end{theorem}

\begin{remark}
When considering $k=3$ and $\alpha = \frac{3}{2}$, equation \eqref{DGBO} becomes 
$$
u_t + D_x^{3/2} u_x + \mu \, u^3u_x= 0
$$
which is critical in $L^2_x(\R)$; in other words, $s_3 = 0$. Consequently, Theorem \ref{Teorema do XSB da introdução} improves the local well-posedness in \cite{MR2859931} for this case. Furthermore, Theorem \ref{efeito regularizante nao linear} introduces a new nonlinear smoothing effect in this scenario.
\end{remark}

Since solutions of problem \eqref{DGBO} satisfy, formally, the conservation of mass
$$
M(u(t))=M(u_0) \quad \forall \, t \in [0,T],
$$
and considering that $s_3< 0$ when $\alpha > \frac{3}{2}$, a straightforward consequence of Theorem \ref{Teorema do XSB da introdução} is:

\begin{corollary}\label{Teorema do XSB da introdução GLOBAL}
Let $\alpha\in (\frac{3}{2},2]$ and $k=3$. If $u_0 \in L^2_x(\mathbb R)$, then the local solutions given in 
Theorem \ref{Teorema do XSB da introdução} can be extended to any time-interval $[0, T]$ for arbitrary $T>0$ in $C_TL^2_x(\R)$.
\end{corollary}

By applying the first iteration of the I-method, first developed by Colliander, Keel, Staffilani, Takaoka, and Tao in \cite{MR1824796} the global existence result provided by Corollary \ref{Teorema do XSB da introdução GLOBAL} can be extended for negative regularity. Indeed, we have the following result. 


\begin{theorem}[Global well-posedness for $k=3$]\label{BOA COLOCAÇAO GLOBAL I METHOD}
Let $\alpha\in (\frac{3}{2},2]$, $k=3$ and $u_0 \in H^s_x(\R)$ with  $ s_3 <-\frac{(2\alpha -3)^2}{24\alpha - 6}<s \leq 0$. Then, the solution of \eqref{DGBO} with initial data $u_0$  can be extended to any interval of time $[0, T]$ for arbitrary $T>0$. Moreover, the solution $u$ satisfies 
\begin{align*}
\sup_{t \in [0,T]}\|u(t)\|_{H^s_x} \lesssim \left(1+ T \right )^{\rho}  \|u_0\|_{L^2_x},
\end{align*}
for any 
$$ \rho\geq \frac{2s}{(2\alpha-3)^2(3-2\alpha-6s)-12(2\alpha-3)(\alpha+1)s}>0.
$$
\end{theorem}

\begin{remark}
Theorem \ref{BOA COLOCAÇAO GLOBAL I METHOD}  recovers the result obtained by Grünrock-Panthee-Silva for  $\alpha = 2$ in \cite{MR2372424}.
\end{remark}

In the case $ k=2 $, with $\alpha=2$, Kenig-Ponce-Vega in \cite{MR1211741} established local well-posedness in $H^s_x(\R)$ with $s \geq \frac{1}{4}$. Guo \cite{MR2860610} established local well-posedness for \eqref{DGBOl2} with $ \alpha \in (1,2) $ in $ H^s_x(\R) $ for $ s > \frac{3-\alpha}{4} $. Moreover, these results are sharp in the sense that the data-solution map fail to exhibit $C^3$-continuity below this index, making it impossible to achieve the critical index using the techniques presented in this paper.

To conclude the introduction, we will discuss the solutions to the Cauchy problem for the case $\alpha=1$, \textit{i.e.} the generalized Benjamin-Ono equation 
\begin{equation}\label{GBO}
\left\{
	\begin{array}{l}
	u_t +\mathcal{H}u_{xx} + \mu \, u^ku_x= 0, \quad (t,x)  \in \R \times \R,\\
		u(0,x)=u_0(x)
	\end{array}
	\right.
\end{equation}
Here, $u=u(t,x)$ is real-valued, $k \in \mathbb{Z}^{+}$, $\mu \in \{\pm 1\}$, and $\mathcal{H}$ denotes the Hilbert transform (see \eqref{hilberttranform}). As we can observe, our results do not cover this case. However, for $k \geq 4$, we can refer to the local well-posedness results obtained by Vento \cite{MR2581042} for data of any size in $H_x^s(\mathbb{R})$, $s \geq s_k$, and the global well-posedness proved by Molinet and Ribaud \cite{MR2101982} in the energy space $H_x^{1/2}(\mathbb{R})$ for data of any size if $k \geq 5$, and for small data if $k=3,4$ (see also \cite{MR2174975}). Recently, in \cite{MR4009456}, Kim and Kwon obtained a global result for $k=4$ in $H_x^{s}(\mathbb{R})$, $s \geq \frac{1}{2}$. In addition, they considered the defocusing case, \textit{i.e.}, $\mu > 0$ in \eqref{GBO}, and demonstrated scattering if $k$ is even, in a similar spirit to the work on the defocusing generalized KdV equation by Farah, Linares, Pastor, and Visciglia \cite{MR3772197}. It is worth noting that for $k=3$, Vento \cite{MR2501678} proved that the problem \eqref{GBO} is ill-posed for $s < \frac{1}{3}$, which implies that the critical index $s_3 = \frac{1}{6}$ is not attained. Moreover, local well-posedness in $H_x^s(\mathbb{R})$, $s > \frac{1}{3}$, was established without any restriction in the initial data. Additionally, global well-posedness in the energy space is only available for small data (see \cite{MR2581042}). For the case $k=2$, see \cite{MR2219229} and the references therein.




The paper is organized as follows: In Section \ref{Notations and Linear Estimates}, we introduce essential notations and present the linear estimates used in our arguments. In Section 3, we establish local well-posedness for $k \geq 4$ in both the critical and subcritical cases, and derive some of its consequences. Additionally, we present the proof of Theorem \ref{wave operator}, which involves the construction of the wave operator. Finally, in Section \ref{seção caso k=3}, we discuss the case $k=3$, demonstrating both local and global well-posedness, and explore the nonlinear smoothing effects.


 \section{Notations and Linear Estimates}\label{Notations and Linear Estimates}
For any two non-negative quantities $X$ and $Y$, the notation $X\lesssim Y$ implies the existence of an absolute constant $C>0$,  independent of $X$ and $Y$, such that $X \leq CY$. We will denote $X\sim Y$ when both $X \lesssim Y$ and $Y \lesssim X$. The notation $X \ll Y$ is used when $X \gtrsim Y$ does not hold. Additionally, if $\kappa \in \R$, then $\kappa^+$ and $\kappa^-$ represent numbers slightly greater and smaller than $\kappa$, respectively.

For $u \in \Sc(\R)$, where $\Sc(\R)$ denotes the Schwartz class, and $s \in \R$ we will define the \textit{Bessel  and Riesz potentials  of order ${\color{red}-}s$}, respectively, by the multipliers
$$
\mathcal{F}_x({J_x^s u})(\xi) := \langle \xi \rangle^s \mathcal{F}_x{u}(\xi) \,\,\, \text{and}\,\,\, \mathcal{F}_x({D_x^s u})(\xi) := |\xi|^s \mathcal{F}_x{u}(\xi)
$$
where $\langle x \rangle : = (1 + |x|^2)^{\frac{1}{2}}$ and $\F_x$ denotes the usual Fourier transform in variable $x$, \textit{i.e.}
$$
\mathcal{F}_x({f})(\xi)= \widehat{f}(\xi):= \int_\R e^{-i x \cdot \xi} f(x)\, dx,
$$
 We will consider $H^s_x(\R)$ the usual Sobolev space, equipped with the norm
$$
\|u\|_{H^s_x}: =\|J^s_x u \|_{L^2_x}.
$$
Similarly, we define the homogeneous Sobolev space $\dot{W}^{s,p}_x(\R)$ with $1 \leq p \leq \infty$ by the norm
\begin{equation}\label{sobolev homogeneo}
\|u\|_{\dot{W}^{s,p}_x}: =\|D^s_x u \|_{L^p_x}.
\end{equation}

When $p=2$, we denote $\dot{W}^{s,p}_x(\R):= \dot{H}^{s}_x(\R)$. Next, we recall the following estimate for fractional derivatives.

\begin{proposition}[\hspace{-0.02cm}\cite{MR1211741}, Theorem A.8]\label{Leibnez Fracionaria}
Let $r \in (0,1)$ and $r_1,r_2 \in [0, r]$ with $r=r_1+r_2$. Let $1<p,q,p_1,q_1,p_2,q_2 <\infty$ with
$$
\frac{1}{p}= \frac{1}{p_1} + \frac{1}{p_2} \,\,\, \text{and}\,\,\, \frac{1}{q}= \frac{1}{q_1} + \frac{1}{q_2}.
$$
Then
\begin{equation*}
 \| D^{r}_x(fg) -fD^{r}_xg- gD^{r}_xf\|_{L^p_xL^q_t} \lesssim \|D^{r_1}_x f\|_{L^{p_1}_xL^{q_1}_t} \|D^{r_2}_xg\|_{L^{p_2}_xL^{q_2}_t}.
\end{equation*}
\end{proposition}
\qed

Let $\{V_{\alpha}(t)\}_{t \in \R}$ be the linear propagator associated to \eqref{DGBO}, \textit{i.e.} for $f \in \Sc(\R)$
\begin{equation}\label{free evolution group}
\left(V_{\alpha}(t)f\right)(x) := c \int_\R e^{i(x\cdot \xi -t |\xi|^\alpha\xi)} \F_x (f) (\xi) \, d\xi= \mathcal{F}^{-1}_x \left({\exp{(-it|\xi|^\alpha\xi)}\widehat{f}(\xi)}\right)(x).  
\end{equation}
That is, $f(t,x):=\left(V_{\alpha}(t)\varphi\right)(x)$ is the solution of the linear problem
    \begin{equation*}
	\left\{
	\begin{array}{l}
	v_t + D_x^\alpha v_x = 0, \quad (t,x)  \in \R \times \R,\\
		v(0,x)=\varphi(x).
	\end{array}
	\right.
\end{equation*}

Now, we present the fundamentals estimates for the operator $V_\alpha(t)$ used in our arguments.
\begin{theorem}\label{full_Strichartz} Consider $4\leq p \leq \infty$, $2\leq q\leq \infty$ with $(p,q) \neq (\infty, \infty)$ such that
$$
\displaystyle\frac{2}{p}+\frac{1}{q} \leq \frac{1}{2}.$$ 
Let $\gamma \in \R$ be such that
\begin{equation*}
    \gamma = \frac{1}{p} + \frac{\alpha+1}{q}-\frac{1}{2},
\end{equation*}
then
\begin{align}
    \label{inverted_order_Strichartz}
    &\|D_x^{\gamma}V_\alpha(t) f\|_{L^p_xL^q_t} \lesssim \|f\|_{L^2_x}
\end{align}
and
\begin{align*}
   & \left\|D_x^{\gamma} \int V_\alpha(-s) g(s) \, ds \right\|_{L^2_x} \lesssim \|g\|_{L^{p'}_xL^{q'}_t}.
\end{align*}
Moreover, if $4\leq p_j \leq \infty$ and $2\leq q_j\leq \infty$ with $(p_j, q_j) \neq (\infty, \infty)$ satisfy 
$$
\displaystyle\frac{2}{p_j}+\frac{1}{q_j} \leq \frac{1}{2} \quad \text{and}\quad\displaystyle\gamma_j = \frac{1}{p_j} + \frac{\alpha+1}{q_j}-\frac{1}{2},
$$
for $j = 1,2$, then
    \begin{equation*}
        \left\|D^{\gamma_1 + \gamma_2}_x \int_{0}^{t} V_\alpha(t-s) g(s) \, ds \right\|_{L^{p_1}_xL^{q_1}_t}
        \lesssim \|f\|_{L^{p_2'}_xL^{q_2'}_t}.
    \end{equation*}
    
\end{theorem}

\begin{proof}
    The inequality \eqref{inverted_order_Strichartz} already holds if $(p,q,\gamma) = (\infty,2,\alpha/2)$ or $(4,\infty, -1/4)$, i.e.,
    \begin{align} \label{maximal_Strichartz_1}
        \|D_x^{\frac{\alpha}{2}}V_\alpha(t) f\|_{L^{\infty}_x L^2_t} \lesssim \|f\|_{L^2_x},\\
        \label{maximal_Strichartz_2}
        \|D_x^{-\frac{1}{4}}V_\alpha(t) f\|_{L^{4}_x L^{\infty}_t} \lesssim \|f\|_{L^2_x},
    \end{align}
    (see, for instance, \cite[Lemma 2.1]{MR1086966} and \cite[Theorem 2.5]{MR1101221}).
    We also have the dispersive estimate (c.f. {\cite[Corollary 2.3]{MR1086966}})
    \begin{equation*}
        \|V_\alpha(t)f\|_{L^{\infty}_x} \lesssim |t|^{-\frac{1}{\alpha+1}}\|f\|_{L^1_x}, 
    \end{equation*}
    which, combined with mass conservation, gives after interpolation
    \begin{equation}\label{classical_Strichartz}
         \|V_\alpha(t) f\|_{L^q_tL^p_x} \lesssim \|f\|_{L^2_x},
    \end{equation}
    if 
    $$
    \frac{1}{p}+\frac{\alpha+1}{q} = \frac{1}{2}$$
    and $p\geq 2$. Moreover, combining Sobolev Embedding  and \eqref{classical_Strichartz},
    \begin{equation}\label{almost_infinity_Strichartz}
        \|D^{-\frac{1}{2}}V_\alpha(t) f\|_{BMO_{t,x}} \lesssim \|V_\alpha(t)f\|_{L^{\infty}_tL^{2}_x} \lesssim \|f\|_{L^2_x}.
    \end{equation}
 The result then follows from interpolating \eqref{maximal_Strichartz_1}, \eqref{maximal_Strichartz_2} and \eqref{almost_infinity_Strichartz} (see \cite{MR4351127}), together with the usual $TT^*$ arguments and the Christ-Kiselev Lemma with reversed norms \cite[Theorem B]{Burq_Planchon_JFA_2006}.
 



\end{proof}






To estimate the nonlinearity, we consider the space
\begin{equation}
\dot N ^s(I) = \dot W^{s+\frac{1}{2},\frac{4\alpha+2}{3(\alpha+1)}}_x L^{\frac{4\alpha+2}{4\alpha-1}}_t (I \times \mathbb R).
\end{equation}

\begin{lemma}\label{re_Strichartz}
Let $I \subset \R$ be an interval and $s < \frac{1}{2}$. Then,  the following linear estimates hold
\begin{align}\label{strichartz_X}
\|V_\alpha(t)f\|_{\dot X^{s}(I)} &\lesssim \|f\|_{\dot H^s_x},\\ \label{strichartz_S}
\left\|V_\alpha(t)f\right\|_{S^{s_k}(I)} &\lesssim \|f\|_{\dot H^{s_k}_x}.
\end{align}
Furthermore, we have the corresponding inhomogeneous estimates
\begin{align}
\label{21}\left\| \int_I V_\alpha(t-\tau) \partial_xg \, d\tau \right\|_{\dot X^{s}(I)} &\lesssim \|g\|_{\dot N^{s}(I)},\\
\left\| \int_I V_\alpha(t-\tau) \partial_xg \, d \tau \right\|_{S^{s_k}(I)} &\lesssim \|g\|_{\dot N^{s_k}(I)}. \label{22}
\end{align}

\end{lemma}
\begin{proof}
   The estimates \eqref{strichartz_X} and \eqref{strichartz_S} follow directly from Theorem \ref{full_Strichartz}, while \eqref{21} and \eqref{22} are derived from Theorem \ref{full_Strichartz} and the Mikhlin-Hörmander theorem for mixed Lebesgue spaces \cite[Corollary 1]{MR3388786}.

\end{proof}

\section{Well-posedness theory for $k\geq 4$.}

 We start by proving useful nonlinear estimates.
\begin{proposition}\label{prop_nonlinear_estimates} Let $I \subset \R$ be an interval and $s \in (-\frac{1}{2},\frac{1}{2})$. Then
    \begin{align*}
        &\|u^{k+1}\|_{\dot N^{s}(I)} \lesssim \|u\|^k_{S^{s_k}(I)} \|u\|_{\dot X^{s}(I)}.
        \end{align*}
Furthermore  
        \begin{align}\label{desigualde com potencia}
        \nonumber \|u^{k+1}-v^{k+1}\|_{\dot N^{s}(I)} &\lesssim \left(\|u\|_{S^{s_k}(I)}+\|v\|_{S^{s_k}(I)}\right)^k\|u-v\|_{\dot X^{s}(I)}\\
        \nonumber&\quad+\left(\|u\|_{S^{s_k}(I)}+\|v\|_{S^{s_k}(I)}\right)^{k-1}\|u\|_{\dot X^{s}(I)}\|u-v\|_{S^{s_k}(I)}\\
        &\quad+\left(\|u\|_{S^{s_k}(I)}+\|v\|_{S^{s_k}(I)}\right)^{k-1}\|v\|_{\dot X^{s}(I)}\|u-v\|_{S^{s_k}(I)}.
    \end{align}
\end{proposition}
\dem This result directly stems from applying the Proposition \ref{Leibnez Fracionaria} with $r = r_1 = s + \frac{1}{2} $, $r_2 = 0,  f = u, g = u^k $, and the corresponding choices of $ p, q, p_1, q_1, p_2, q_2  \in (1, \infty)$. Additionally, for \eqref{desigualde com potencia} it is enough to recall that
\begin{equation}\label{pointwise_difference_bound}
|u^{k+1}- v^{k+1}| \lesssim |u^k- v^k||u-v|.    
\end{equation}
\qed

\subsection{Local well-posedness in $\dot{H}^{s_k}_x(\R)$}\label{LWP in Critical space}
   
With the previous estimates at hand, we are ready to prove Theorem \ref{Teorema existencia critica}.

\begin{proof}[Proof of Theorem \ref{Teorema existencia critica}: ] 
    For $\varepsilon_0>0$ to be chosen later, if $c>0$ is the implicit constant in Theorem \ref{full_Strichartz}, define
    \begin{align*}
        E_T := \big\{u \in \dot X^{s_k}_T \cap S^{s_k}_T \,:\, \|u\|_{S^{s_k}_T} \leq 2 \|V_\alpha(t)u_0\|_{S^{s_k}_T}\,\,\text{and}\,\,\|u\|_{\dot X^{s_k}_T}\leq 2c \|u_0\|_{\dot H^{s_k}_x} \big\},
    \end{align*}
    equipped with the metric
    \begin{equation*}
        \rho(u,v) = \|u-v\|_{\dot X^{s_k}_T} +  \|u-v\|_{S^{s_k}_T}.
    \end{equation*}
 We set
    \begin{equation*}
        \Phi(u) := V_\alpha(t)u_0 +\mu \int_0^t V_\alpha(t-\tau)\partial_x(u^{k+1})(\tau) \, d\tau 
    \end{equation*} 
    which defines a contraction on $E_T$. Indeed, by Remark \ref{re_Strichartz}, Proposition \ref{prop_nonlinear_estimates}, and \eqref{time_well_posedness}, if $u \in E_T$ is such that \eqref{time_well_posedness} holds, we have
\begin{align}
      \nonumber \|\Phi(u)\|_{S^{s_k}_T} & \leq \|V_\alpha(t)u_0\|_{S^{s_k}_T} + \|u\|_{S^{s_k}_T}^k \|u\|_{\dot X^{s_k}_T}\\
       &\leq \left(1+2^kc\varepsilon_0^{k-1}\|u_0\|_{\dot H^{s_k}_x}\right)\|V_\alpha(t)u_0\|_{S^{s_k}_T}
\end{align}
and
\begin{align*}
        &\|\Phi(u)\|_{\dot X^{s_k}_T} \leq c\|u_0\|_{\dot H^{s_k}_x} + c\|u\|_{S^{s_k}_T}^k \|u\|_{\dot X^{s_k}_T} \leq \left(1+2^k\varepsilon_0^{k}\right)c\|u_0\|_{\dot H^{s_k}_x}.
\end{align*}
Therefore
\begin{equation*}
\rho(\Phi(u),\Phi(v))\leq 4^k\varepsilon_0^{k-1}\left(\varepsilon_0+\|u_0\|_{\dot H^{s_k}}\right)c\rho(u,v).
\end{equation*}
Thus, it is suffices to choose $\displaystyle0<\varepsilon_0^{k-1}<\min \left\{\frac{1}{4^{k+1}c},\frac{1}{4^{k+1}c\|u_0\|_{\dot H^{s_k}}}\right\}$.

Now, if $u$ and $\tilde u$ are solutions with initial data $u_0$ and $\tilde u_0$, we first estimate
\begin{equation*}
    \|V_\alpha(t)\tilde u_0\|_{S^{s_k}_T} \leq \|V_\alpha(t)u_0\|_{S^{s_k}_T} + O(\|u_0 - \tilde u_0\|_{\dot H^{s_k}}) 
    \leq \epsilon_0 + O(\|u_0 - \tilde u_0\|_{\dot H^{s_k}}),
\end{equation*}
which says we are able to choose the same $T>0$ for both solutions in the fixed-point argument, given $u_0$ and $\tilde u_0$ are close in $\dot H^{s_k}$ (say, $\|u_0 - \tilde u_0\|_{\dot H^{s_k}} \ll \epsilon_0$). Now, from the Duhamel formula and the  Strichartz estimates, we deduce
\begin{align*}
    \rho(u,\tilde u) 
    &\lesssim \|u_0 - \tilde u_0\|_{\dot H^{s_k}} + \left[\|u\|_{S^{s_k}} + \|\tilde u\|_{S^{s_k}}\right]^k \rho(u,\tilde u)
    \\&\quad+
    \left[\|u\|_{S^{s_k}} + \|\tilde u\|_{S^{s_k}}\right]^{k-1}
    (\|u\|_{\dot X^{s_k}} + \|\tilde u\|_{\dot X^{s_k}})    
    \rho(u,\tilde u)
    \\
    &\lesssim  \|u_0 - \tilde u_0\|_{\dot H^{s_k}} + \epsilon_0^{k-1}\left[\epsilon_0 + \|u_0\|_{\dot H^{s_k}}\right] \rho(u,\tilde u)
\end{align*}

since $k > 1$, we have
\begin{equation*}
\rho(u,\tilde u) \lesssim  \|u_0 - \tilde u_0\|_{\dot H^{s_k}}.
\end{equation*}

Therefore, the continuity of solutions with respect to initial data follows.
\end{proof}


\begin{proof}[Proof of Corollary {\ref{cor:critical_small_data}}]
    Recalling the Strichartz estimate \eqref{strichartz_S}, from the proof of Theorem \ref{Teorema existencia critica}, it is enough to choose $\delta_{sd} = \min\left\{\frac{1}{4c},1\right\}$. We then have
    \begin{equation*}
        \|u\|_{S^{s_k}(\mathbb R)} +
        \|u\|_{\dot X^{s_k}(\mathbb R)} \lesssim \|u_0\|_{\dot H^{s_k}_x} < \delta_{sd},
    \end{equation*}
    which allows us to define 
    \begin{equation*}
        \phi_+ = u_0 + \mu\int_0^{+\infty} V_\alpha(-\tau) \partial_x (u^{k+1}(\tau))\, d\tau 
    \end{equation*}
    and
    \begin{equation*}
        \phi_- = u_0 + \mu\int_{-\infty}^{0} V_\alpha(-\tau) \partial_x (u^{k+1}(\tau))\, d\tau 
    \end{equation*}
    belonging to $\dot H^{s_k}_x(\R)$, since
    \begin{equation*}
        \|\phi_{\pm}\|_{\dot H^{s_k}_x} \lesssim \|u_0\|_{\dot H^{s_k}_x} + \|u_0\|_{\dot H^{s_k}_x}^{k+1}. 
    \end{equation*}
    Moreover, 
    \begin{equation*}
        \|u(t) - V_\alpha(t)\phi_{+}\|_{\dot H^{s_k}_x} \leq \|u\|_{S^{s_k}([t,\infty))}^k \|u_0\|_{\dot H^{s_k}_x} \to 0,
    \end{equation*}
    as $t \to +\infty$, with an analogous assertion for $\phi_-$.
\end{proof}

We can also derive the first case of Corollary \ref{cor:subcritical_blowup_rate}.

\begin{proof}[Proof of Corollary {\ref{cor:subcritical_blowup_rate}}: Case 1.]
    Suppose $\|u\|_{S^{s_k}_{T^*}} < +\infty$. For $0<s<t<T^*$, we have
    \begin{equation*}
        \|u\|_{\dot X^{s_k}([s,t])} \lesssim \|u(s)\|_{\dot H^{s_k}_x} + \|u\|^k_{S^{s_k}([s,T^*))} \|u\|_{\dot X^{s_k}([s,t])}.
    \end{equation*}
    This implies that there exists $s_0<T^*$ such that
    \begin{equation*}
        \|u\|_{\dot X^{s_k}([s_0,T^*))} \lesssim \|u(s_0)\|_{\dot H^{s_k}_x},
    \end{equation*}
    that is, $\|u\|_{\dot X^{s_k}_{T^*}} < +\infty$. Moreover, for $0<s<t<T^*$,
    \begin{equation*}
    \|u(t)-u(s)\|_{\dot H^{s_k}_x} \lesssim \|[V_\alpha(t-s)-I]u_0\|_{\dot H^{s_k}_x} +     \|u\|^k_{S^{s_k}([s,t])} \|u\|_{\dot X^{s_k}_{T^*}}
    \end{equation*}
    from which we deduce that
    $$
    \lim_{t \nearrow T^*} u(t)
    $$
exists in $\dot H^{s_k}_x(\R)$. By uniqueness, this contradicts the maximality of $T^*$.

\end{proof}

Analogous estimates as above in a suitable space allow us to prove Theorem \ref{teo:propagation_regularity}. Indeed,


\begin{proof}[Proof of Theorem {\ref{teo:propagation_regularity}}] It is enough to run a fixed-point argument on
\begin{align*}
        E_T = \big\{u \in \dot X^{s_k}_T \cap \dot X^{s_k}_T\cap S^{s_k}_T \,:\, \|u\|_{S^{s_k}_T} \leq 2 \|V_\alpha(t)u_0\|_{S^{s_k}_T},
        \,\, &\|u\|_{\dot X^{s}_T}\leq 2c \|u_0\|_{\dot H^{s}_x},
        \\ &\|u\|_{\dot X^{s_k}_T}\leq 2c \|u_0\|_{\dot H^{s_k}_x} \big\},
    \end{align*}
    equipped with the metric
    \begin{equation*}
        \rho(u,v) = \|u-v\|_{\dot X^{s}_T}+ \|u-v\|_{\dot X^{s_k}_T} +  \|u-v\|_{S^{s_k}_T}.
  \end{equation*}
    
\end{proof}

\subsection{Local well-posedness in $\dot H^s_x(\R)$ for $s>s_k$} \label{LWP in subcritical space}


We are also interested in the well-posedness of the problem \eqref{DGBO} for data in homogeneous Sobolev spaces. Since we cannot rely on the critical-scaling norm $S^{s_k}$ in this context, we need to replace it with a suitable alternative.

First of all, we observe that by Theorem \ref{full_Strichartz}, we have the following estimates
\begin{equation*}
 \|V_\alpha(t)f\|_{S^s(I)} \lesssim \|f\|_{\dot H^s_x(\R)}
\end{equation*}
and
\begin{align*}
\left\|\int V_\alpha(t-\tau)\partial_xg \, d \tau \right\|_{S^s(I)} \lesssim \|g\|_{\dot N^{s}(I)}.    
\end{align*}
Moreover, if we consider $I=[0,T]$ we obtain by the H\"older inequality
\begin{align*}
        \|u^{k+1}\|_{\dot N^{s}_T} &\lesssim T^ {\frac{k(s-s_k)}{\alpha+1}}\|u\|^k_{S_{T}^s} \|u\|_{\dot X^{s}_T}.
\end{align*}
Similarly,
\begin{align*}
        \|u^{k+1}-v^{k+1}\|_{\dot N^{s}_T} &\lesssim T^ {\frac{k(s-s_k)}{\alpha+1}}\left(\|u\|_{S_{T}^s}+\|u\|_{S_{T}^s}\right)^k\|u-v\|_{\dot X^{s}_T} \\&\hspace{0.4cm} +T^ {\frac{k(s-s_k)}{\alpha+1}}\left(\|u\|_{S_{T}^s}+\|u\|_{S_{T}^s}\right)^{k-1}(\|u\|_{\dot X^{s}_T} + \|u\|_{\dot X^{s}_T})\|u-v\|_{S_{T}^s}.
    \end{align*}

The above estimates allows us to derive Theorem {\ref{teo:subcritical_lwp}}.

    
\begin{proof}[Proof of Theorem {\ref{teo:subcritical_lwp}}] By employing similar arguments as in the proof of Theorem \ref{Teorema existencia critica}, it is enough to choose $T \displaystyle\ll {\|u_0\|_{\dot H^s_x}^{-\frac{\alpha+1}{s-s_k}}}$ and to consider the space
\begin{align}\label{fixed_point_space_homgeneous_LWP}
        E_T = \big\{u \in \dot X^{s}_T \cap S^s_T \,:\,& \|u\|_{S^s_T}  +\|u\|_{\dot X^{s}_T}\leq 2c \|u_0\|_{\dot H^{s}_x} \big\},
    \end{align}
    equipped with the metric
    \begin{equation*}
        \rho(u,v) = \|u-v\|_{\dot X^{s}_T} +  \|u-v\|_{S^s_T}.
    \end{equation*}
\end{proof}

A more refined analysis can provide us with the desired subcritical blow-up rate stated in Corollary {\ref{cor:subcritical_blowup_rate}}.


    
\begin{proof}[Proof of Corollary {\ref{cor:subcritical_blowup_rate}}: Case 2.]
    Suppose first that $\|u\|_{S_{T^*}^s } < +\infty$. Similarly to the proof of Case 1, this implies $\|u\|_{\dot X^s_{T^*} } < +\infty$, which in turn contradicts the maximality of $T^*$.
    
    Assume now that there exists a sequence of times $t_n \nearrow T^*$ such that
    
\begin{equation*}
    (T^* -t_n)^{(s-s_k)/(\alpha+1)}\|u(t_n
)\|_{\dot H^s_x} \leq \frac{1}{n}.
\end{equation*}
For $t \in (t_n,T^*)$, we have

\begin{equation*}
    0\leq t-t_n \leq  \frac{1}{(n\|u(t_n)\|_{\dot H^s_x})^{\frac{\alpha+1}{s-s_k}}}.
\end{equation*}
Thus, recalling \eqref{fixed_point_space_homgeneous_LWP} we conclude
\begin{equation*}
    \|u\|_{\dot X^s([t_n,t])} \lesssim \|u(t_n)\|_{\dot H^s_x} + \left[(t-t_n)^{(s-s_k)/(\alpha+1)}\|u(t_n)\|_{\dot H^s_x}\right]^k \|u\|_{\dot X^s([t_n,t])}.
\end{equation*}
Therefore, there exists $n_0$ such that 
\begin{equation*}
    \|u\|_{\dot X^s([t_{n_0},T^*))} \lesssim \|u(t_{n_0})\|_{\dot H^s_x}. 
\end{equation*}
The last estimate, in turn, implies, in the same fashion, that
\begin{equation*}
    \|u\|_{ S^s([t_{n_0},T^*))} \lesssim \|u(t_{n_0})\|_{\dot H^s_x},
\end{equation*}
which contradicts \eqref{infinite_subcritical_scattering_norm}.
\end{proof}

If the initial data $u_0$ is assumed to belong to an inhomogeneous Sobolev space $H^s_x(\mathbb{R})$, \textit{i.e.}, has finite mass, the fixed-point argument can be applied using metrics defined by the $\dot{X}^0(I) \cap S^0(I)$ norm, which imposes no upper restriction on $s \geq s_k$. However, we will omit the proof for this case.

\subsection{Construction of the wave operator}\label{Proof of Wave Operator}

To prove the Theorem \ref{wave operator}, we will closely follow the techniques developed in \cite{MR3043074}. More precisely, we will show the existence of a fixed point for the operator
$$
\Psi(v)(t):= -\mu \int_{t}^\infty V_\alpha(t- \tau)\partial_x \left( v(\tau)+ V_\alpha(\tau)v_0 \right)^{k+1}\, d\tau,
$$
considered on the time interval $[T_0, \infty)$, where $T_0 > 0$ is a sufficiently large number to be determined later, and with $v_0 \in \dot{H}^{s_k}_x(\R)$. By applying similar arguments to those in \cite{MR2384545}, we obtain the following proposition.

\begin{proposition}[\hspace{-0.02cm}\cite{MR2384545}, Proposition 4.1]\label{farah}
Let $v$ be a fixed point of the operator $\Psi$. Then the function
\begin{equation}
u(t) = V_\alpha(t)v_0 + v(t),
\end{equation}
satisfies the Duhamel formula of the equation  \eqref{DGBO} in the time interval $[T_0, \infty)$.
\end{proposition}

In order to prove Theorem \ref{wave operator}, we apply the contraction mapping principle. Given $T>0$ and $I = [T,+\infty)$ we consider the space
\begin{align*}
    F_T = \left\{ v \in \dot X^{s_k}(I) \cap S^{s_k}(I) :
    \|v\|_{\dot X^{s_k}(I)} \leq 2 c \|v_0\|_{\dot H^{s_k}},\,\|v\|_{S^{s_k}(I) } \leq 2 \|V_\alpha(t)v_0\|_{S^{s_k}(I)} \right\}.
\end{align*}

Observe that, from \eqref{22}, we have
\begin{equation*}
\|\Psi (v) \|_{S^{s_k}(I)} \lesssim \| (v+ V_\alpha(t)v_0 )^{k+1}) \|_{\dot N^{s_k}(I)}.
\end{equation*}

By Proposition \ref{prop_nonlinear_estimates} we obtain
\begin{align*}
\|\left( v+V_\alpha(t)v_0 \right)^{k+1} \|_{\dot N^{s_k}(I)}&\lesssim \|v+ V_\alpha(t)v_0\|_{S^{s_k}(I)}^k  \| v+ V_\alpha(t)v_0\|_{\dot X^{s_k}(I)}.
\end{align*}

Therefore, by Theorem \ref{full_Strichartz} we can conclude
\begin{align*}
\|\Psi (v) \|_{\dot X^{s_k}(I) \cap S^{s_k}(I)} & \lesssim\|v+V_\alpha(t)v_0\|_{S^{s_k}(I)}^k\|\left( v+ V_\alpha(t)v_0\right)\|_{\dot X^{s_k}(I)}\\
&\lesssim \|V_\alpha(t)v_0\|_{S^{s_k}(I)}^k \| v_0\|_{\dot{H}^{s_k}_x}.  
\end{align*}
Since $\|V_\alpha(t)v_0\|_{S^{s_k}(\mathbb R)} < + \infty$ one can choose $T_0>0$ large enough such that
$$
\Psi: F_T \longrightarrow F_T,
$$
is well-defined. We can also see that $\Phi$ defines a contraction since, for $v$, $w \in F_T$,
\begin{align*}
    \|\Psi (w_1) - &\Psi (w_2)  \|_{\dot X^{s_k}(I) \cap S^{s_k}(I)}\\ &\lesssim \left(\|w_1\|_{S^{s_k}(I)}+\|w_2\|_{S^{s_k}(I)}+\|V_\alpha(t)v_0\|_{S^{s_k}(I)}\right)^k\| w_1 -w_2\|_{\dot X^{s_k}(I)}\\
&\lesssim \|V_\alpha(t)v_0\|_{S^{s_k}(I)}^k \| w_1-w_2\|_{\dot{H}^{s_k}_x} .
\end{align*}

By the contraction mapping principle, $\Psi$ has a unique fixed point $v \in F_T$. Thus, by Proposition \ref{farah}, the function
$$
u(t):= v(t) + V_\alpha(t)v_0,
$$
satisfies the integral equation on the time interval $[T_0,\infty)$. Hence,
\begin{align*}
\|u(t) - V_\alpha(t)v_0\|_{\dot{H}^{s_k}_x} &= \|v(t)\|_{\dot{H}^{s_k}_x}\\
&\leq \|\Psi(v)\|_{L^\infty_t \dot{H}^{s_k}_x([t,+\infty))}\\
&\lesssim \|V_\alpha(t)v_0\|_{S^{s_k}([t,+\infty))}^k\| v_0\|_{\dot H^{s_k}} \to 0 \quad \text{as } t \to +\infty.
\end{align*}
This finishes the proof.

\qed   

\section{Case $k=3$}\label{seção caso k=3}

In this section, we will examine problem \eqref{DGBO} for $k=3$, that is,
   \begin{equation}\label{DGBOl2}
	\left\{
	\begin{array}{l}
	u_t + D_x^{\alpha} u_x +  \mu \, u^3u_x= 0, \quad (t,x)  \in \R \times \R,\\
		u(0,x)=u_0(x).
	\end{array}
	\right.
\end{equation}
We start by pointing out that the method used in the previous section does not apply in this case. To treat problem \eqref{DGBOl2}, we employ a fixed-point argument based on the one introduced by Bourgain to study dispersive equations (see \cite{MR1209299} and \cite{MR1215780}).

For $s,b \in \R$ we define the  Bourgain space $X^{s,b}_\alpha := X^{s,b}_\alpha(\R \times \R)$ related to the equation in \eqref{DGBOl2} as the completion of the Schwartz class $\Sc(\R^2)$ under the norm
$$
\|u\|_{X^{s,b}_\alpha}^2 := \int_\R \int_\R \langle  \tau- \xi|\xi|^{\alpha}\rangle^{2b} \langle \xi \rangle^{2s} |\tilde{u}(\tau, \xi)|^2 \, d\tau d\xi 
$$
where  $\tilde{u}(\tau,\xi):= \F_{t,x}(u)(\tau, \xi)$ denotes the space-time Fourier transform of $u$. Also, consider a smooth function  $\psi \in C^{\infty}_c(\R)$ with $\supp \psi \subset [-2,2]$ and $\psi \equiv 1$ in $[-1,1]$. Denote
$$
\psi_T(t):= \psi\left(\frac{t}{T}\right), \quad T \in (0,1).
$$

Our goal in this section is to establish the following result.
\begin{theorem}\label{Teorema em XSB}
Let $1<\alpha \leq 2$ and $s>s_3$. Given $u_0 \in H^s_x(\R)$, there exist $T>0$ and a unique $u \in X^{s,\frac{1}{2}^+}_\alpha$ satisfying
$$
u(t)= \psi(t)V_\alpha(t) u_0 + \mu \,\, \psi_T(t) \int_0^t V_\alpha(t-\tau)\partial_x(u^4)(\tau)\, d\tau.
$$
\end{theorem}
In particular, since $ X^{s,\frac{1}{2}^+}_\alpha \hookrightarrow C_t H^s_x $, we obtain Theorem \ref{Teorema do XSB da introdução}.

Proving local well-posedness for \eqref{DGBOl2} using a Picard iteration scheme in Bourgain spaces $X^{s,b}_\alpha$ depends on deriving a multilinear estimate of the form
\begin{equation}\label{estimativa multilinear bourgain}
\|\partial_x(u^4)\|_{X^{s,b-1}_\alpha} \lesssim \|u\|_{X^{s,b}_\alpha}^4,
\end{equation}
for $b>1/2$. In order to obtain this estimate, we will use the techniques developed in Correia-Oliveira-Drummond in \cite{restrictionmethod}. To present the main results from \cite{restrictionmethod} needed in our work, we consider the generic dispersive PDE
$$
u_t - i L(D)u =N(u),
$$
where $L(D)$, is a linear (pseudo-) differential operator in the spatial variables given by a real Fourier symbol $L(\xi)$, \textit{i.e.}
$$
{L(D) f}(x):= \left(L(\xi) \widehat{f}(\xi)\right)^\vee(x),
$$
and $N(u)$ is a general nonlinearity defined by the multilinear form $N(u)= N(u,...,u)$ where
$$
\mathcal{F}\left[N(f_1,...,f_k)\right] (\xi) = \int_{\xi =\sum_{j=1}^k \xi_j} m(\xi_1, ...,\xi_k) \widehat{f_1}(\xi_1) \cdots \widehat{f_k}(\xi_k) \, d\xi_1 \cdots d\xi_{k-1}.
$$
It is also necessary to consider the  Bourgain space $X^{s,b}_{\tau=L(\xi)}$ related to the symbol $L(\xi)$ defined by the norm
$$
\|u\|_{X^{s,b}_{\tau=L(\xi)}}^2 := \int_\R \int_\R \langle  \tau- L(\xi)\rangle^{2b} \langle \xi \rangle^{2s} |\tilde{u}(\tau, \xi)|^2 \, d\tau d\xi.
$$
More precisely, in this conditions, we use the following lemmas.

\begin{lemma}[\hspace{0.02cm}\cite{restrictionmethod}, Lemma 1]\label{Lema 1}
 Suppose that  there exists $\gamma < -b^\prime$ such that, for any $M>1$,
\begin{equation}\label{FRM Cauchy-Schwarz}
\sup_{\xi} \int_{\xi=\sum_{j=1}^k \xi_j}  \frac{|m(\xi_1, ...,\xi_k)|^2 \langle \xi \rangle^{2s^\prime}}{\langle \xi_1 \rangle^{2s}\langle \xi_2 \rangle^{2s}\cdots\langle \xi_k \rangle^{2s}} \, \1_{|\Phi|<M}\, d\xi_1 \cdots d \xi_{k-1} \lesssim M^{2 \gamma},
\end{equation}
with $\Phi(\xi_1,...,\xi_k) := L(\xi) - \sum_{j=1}^k L(\xi_j)$. Then
\begin{equation}\label{estimativa |N|XSB}
\|N(u_1,...,u_k)\|_{X^{s^\prime,b^\prime}_{\tau=L(\xi)}} \lesssim \prod_{j=1}^k \|u_j\|_{X^{s,b}_{\tau=L(\xi)}}.
\end{equation}
\end{lemma}
\begin{lemma}[\hspace{0.02cm}\cite{restrictionmethod}, Lemma 2]\label{Lema 2}Here we adopt the convention that $\xi_0 = \xi$. 
Suppose that there exist $ \emptyset \neq A \subsetneq \{0,1,2,...,k\}$, $\M_j(\xi_1,...,\xi_k) \geq 0$, $j=1,2$, and $\gamma < -2b^\prime$ such that
$$
\M_1 \M_2 = \frac{|m(\xi_1, ...,\xi_k)|^2 \langle \xi \rangle^{2s^\prime}}{\langle \xi_1 \rangle^{2s}\cdots\langle \xi_k \rangle^{2s}}
$$
and, for any $M>1$,
\begin{equation}\label{estimativafrequenciarestrita}
\sup_{\xi_j \in A, \beta} \int_{\xi=\sum_{j=1}^4 \xi_j} \M_1 \1_{|\Phi -\beta|< M} d\xi_{j \notin A} + \sup_{\xi_j \notin A, \beta} \int_{\xi=\sum_{j=1}^4 \xi_j} \M_2 \1_{|\Phi -\beta|< M} d\xi_{j \in A}  \lesssim M^\gamma,
\end{equation}
with $\Phi(\xi_1,...,\xi_k) := L(\xi) - \sum_{j=1}^k L(\xi_j)$. Then the estimate \eqref{estimativa |N|XSB} follows.
\end{lemma}
\begin{remark}
The notation $d\xi_{j \notin A}$ and $d\xi_{j \in A}$ indicates integration over all variables $\xi_j$ not in, and in, the set $A$, respectively.
\end{remark}
\textit{Proof of Theorem \ref{Teorema em XSB}}: As we commented before we will focus in the proof of estimate \eqref{estimativa multilinear bourgain}. Therefore, let $b = \frac{1}{2}^+ $, $b' = (b - 1)^+ $, and $ s = s' $.  In terms of the Lemmas \ref{Lema 1} and \ref{Lema 2} we have
$$
L(\xi):= \xi|\xi|^\alpha, \quad \text{for} \,\, 1< \alpha \leq 2,
$$
with the phase function $\Phi= \xi|\xi|^\alpha -\sum_{j=1}^4 \xi_j|\xi_j|^\alpha$. The nonlinearity, expressed in the multilinear terms, is defined by
$$
\mathcal{F}\left[N(u_1,u_2,u_3,u_4)\right] (\xi) = -i \xi \int_{\xi =\sum_{j=1}^4 \xi_j}  \widehat{u_1}(\xi_1) \widehat{u_2}(\xi_2) \widehat{u_3}(\xi_3)  \widehat{u_4}(\xi_4) \, d\xi_1 d \xi_2 d\xi_3,
$$
and $X^{s,b}_{\tau= \xi |\xi|^\alpha}\equiv X^{s,b}_\alpha$. We demonstrate that the estimates \eqref{FRM Cauchy-Schwarz} and \eqref{estimativafrequenciarestrita} holds for the problem \eqref{DGBOl2} by employing the techniques introduced in \cite{restrictionmethod} and \cite{MR4097938}. 

For that, without loss of generality, by symmetry, we can assume $|\xi| \geq |\xi_1| \geq |\xi_2| \geq |\xi_3| \geq |\xi_4|$. If $|\xi| \leq 1$ we have
$$
\frac{|\xi|^2 \langle \xi \rangle^{2s}}{\langle \xi_1 \rangle^{2s}\langle \xi_2 \rangle^{2s}\langle \xi_3 \rangle^{2s}\langle \xi_4 \rangle^{2s}} \lesssim 1.
$$
Since $\{(\xi_1,\xi_2,\xi_3,\xi_4) \in \R^4: \xi = \sum_{j=1}^4 \xi_j \}$ is a bounded region whenever $|\xi| \leq 1$, then
$$
    \int_{\xi=\sum_{j=1}^4 \xi_j}  \frac{|\xi|^2 \langle \xi \rangle^{2s}}{\langle \xi_1 \rangle^{2s}\langle \xi_2 \rangle^{2s}\langle \xi_3 \rangle^{2s}\langle \xi_4 \rangle^{2s}} \, \1_{|\Phi|<M}\, d\xi_1 d\xi_2 d \xi_3 \lesssim 1< M^{1^-}.
$$
By applying Lemma \ref{Lema 1}, the estimate \eqref{estimativa multilinear bourgain} follows in this region.

Now, consider $|\xi|>1$. In this scenario we have
$$
|\xi_1| \sim |\xi|\sim \langle \xi \rangle.
$$
We split the proof in some cases.

\textbf{Case 1.} The scenario $|\xi| \simeq |\xi_1| \simeq |\xi_2| \simeq |\xi_3| \gg |\xi_4|$ does not occur: In this case, consider $A=\{1,2,4\}$ and define
$$
\M_1(\xi_1,\xi_2,\xi_3,\xi_4) := \frac{|\xi| \langle \xi \rangle^{s+1/2}}{\langle \xi_1 \rangle^{s+1/2}\langle \xi_2 \rangle^{s+1/2}\langle \xi_3 \rangle^{s}\langle \xi_4 \rangle^{s}}
$$
and
$$
\M_2(\xi_1,\xi_2,\xi_3,\xi_4) := \frac{|\xi| \langle \xi \rangle^{s-1/2}}{\langle \xi_1 \rangle^{s-1/2}\langle \xi_2 \rangle^{s-1/2}\langle \xi_3 \rangle^{s}\langle \xi_4 \rangle^{s}}.
$$
We begin by estimating the term
\begin{align*}
I_1(\xi, \xi_3) &=  \int \M_1 \1_{|\Phi -\beta|< M}\,  d\xi_1 d\xi_2 \\
 &=  \int  \frac{|\xi| \langle \xi \rangle^{s+1/2}}{\langle \xi_1 \rangle^{s+1/2}\langle \xi_2 \rangle^{s+1/2}\langle \xi_3 \rangle^{s}\langle \xi_4 \rangle^{s}} \1_{|\Phi -\beta|< M}\,  d\xi_1 d\xi_2,
\end{align*}
with $\xi, \xi_3$ fixed and $\xi_4 = \xi -\xi_1 - \xi_2 -\xi_3$. We need to consider two additional subcases.

\textbf{Subcase 1.1.} $|\partial_{\xi_1} \Phi| \gtrsim |\xi|^{\alpha}$: Suppose first that $s \leq 0$. Then, since $|\xi_2| \geq |\xi_3| \geq |\xi_4|$ and $|\xi| \sim |\xi_1|$, by the Hölder inequality we obtain
\begin{align*}
I_1 &=  \int  \frac{|\xi| \langle \xi \rangle^{s+1/2}}{\langle \xi_1 \rangle^{s+1/2}\langle \xi_2 \rangle^{s+1/2}\langle \xi_3 \rangle^{s}\langle \xi_4 \rangle^{s}} \1_{|\Phi -\beta|< M}\,  d\xi_1 d\xi_2 \\
&\lesssim  \int  \frac{|\xi| }{\langle \xi_2 \rangle^{s+1/2}\langle \xi_3 \rangle^{s}\langle \xi_4 \rangle^{s}} \1_{|\Phi -\beta|< M}\,  d\xi_1 d\xi_2 \\
&\lesssim  \int  \frac{|\xi| }{\langle \xi_2 \rangle^{3s+1/2}} \1_{|\Phi -\beta|< M}\,  d\xi_1 d\xi_2 \\
&\lesssim  \left( \int  \frac{|\xi|^{1^+} }{\langle \xi_2 \rangle^{3s+1/2^{-}}} \1_{|\Phi -\beta|< M}\,  d\xi_1 d\xi_2 \right)^{1^-}.
\end{align*}
By performing the variable change $\xi_1 \mapsto \Phi$, we obtain
\begin{align*}
I_1 &\lesssim \left( \int  \frac{|\xi|^{1^+} }{\langle \xi_2 \rangle^{3s+1/2^{-}}} \1_{|\Phi -\beta|< M}\,  \frac{d\Phi}{|\partial_{\xi_1} \Phi|} d\xi_2 \right)^{1^-}\\
&\lesssim \left( \int  \frac{|\xi|^{(1-\alpha)^+} }{\langle \xi_2 \rangle^{3s+1/2^{-}} } \1_{|\Phi -\beta|< M}\,  d\Phi d\xi_2 \right)^{1^-} \\
&\lesssim \left( \int  \frac{|\xi|^{(1-\alpha)^+}}{\langle \xi_2 \rangle^{3s+1/2^{-}}} \left[\int\1_{|\Phi -\beta|< M}\,  d\Phi \right] d\xi_2 \right)^{1^-} \\
&\lesssim M^{1^-} \cdot \left( \int  \frac{|\xi|^{(1-\alpha)^+}}{\langle \xi_2 \rangle^{3s+1/2^{-}}} d\xi_2 \right)^{1^-} .
\end{align*}
Now consider $s > 0$. For this case, since that $\langle \xi_3 \rangle^{-s}\langle \xi_4 \rangle^{-s} \lesssim 1$ we have
\begin{align*}
I_1 &=  \int  \frac{|\xi| \langle \xi \rangle^{s+1/2}}{\langle \xi_1 \rangle^{s+1/2}\langle \xi_2 \rangle^{s+1/2}\langle \xi_3 \rangle^{s}\langle \xi_4 \rangle^{s}} \1_{|\Phi -\beta|< M}\,  d\xi_1 d\xi_2 \\
&\lesssim  \int  \frac{|\xi| }{\langle \xi_2 \rangle^{s+1/2}\langle \xi_3 \rangle^{s}\langle \xi_4 \rangle^{s}} \1_{|\Phi -\beta|< M}\,  d\xi_1 d\xi_2 \\
&\lesssim  \int  \frac{|\xi| }{\langle \xi_2 \rangle^{s+1/2}} \1_{|\Phi -\beta|< M}\,  d\xi_1 d\xi_2 \\
&\lesssim  \left( \int  \frac{|\xi|^{1^+} }{\langle \xi_2 \rangle^{s+1/2^{-}}} \1_{|\Phi -\beta|< M}\,  d\xi_1 d\xi_2 \right)^{1^-}.
\end{align*}
Again, by making the variable change $\xi_1 \mapsto \Phi$, we obtain
\begin{align*}
I_1 &\lesssim \left( \int  \frac{|\xi|^{1^+} }{\langle \xi_2 \rangle^{s+1/2^{-}}} \1_{|\Phi -\beta|< M}\,  \frac{d\Phi}{|\partial_{\xi_1} \Phi|} d\xi_2 \right)^{1^-}\\
&\lesssim \left( \int  \frac{|\xi|^{(1-\alpha)^+} }{\langle \xi_2 \rangle^{s+1/2^{-}} } \1_{|\Phi -\beta|< M}\,  d\Phi d\xi_2 \right)^{1^-} \\
&\lesssim \left( \int  \frac{|\xi|^{(1-\alpha)^+}}{\langle \xi_2 \rangle^{s+1/2^{-}}} \left[\int\1_{|\Phi -\beta|< M}\,  d\Phi \right] d\xi_2 \right)^{1^-} \\
&\lesssim M^{1^-} \cdot \left( \int  \frac{|\xi|^{(1-\alpha)^+}}{\langle \xi_2 \rangle^{s+1/2^{-}}} d\xi_2 \right)^{1^-} .
\end{align*}
Therefore, in both cases, since $1<\alpha \leq 2$ and $s> s_3$
$$
\sup_{\xi, \xi_3, \beta} I_1(\xi, \xi_3) \lesssim M^{1^-}.
$$

\textbf{Subcase 1.2.}  $|\partial_{\xi_1} \Phi| \ll |\xi|^{\alpha}$: Since that
$$
\partial_{\xi_1} \Phi \simeq |\xi_4|^{\alpha} - |\xi_1|^{\alpha},
$$
this implies that
$$
||\xi_4|^{\alpha} - |\xi_1|^{\alpha}|\simeq|\partial_{\xi_1} \Phi| \ll |\xi|^{\alpha} \sim |\xi_1|^{\alpha},
$$
that is, $|\xi_1| \simeq |\xi_4|$, which, by ordering, ensures that
$$
|\xi_i| \simeq |\xi_j|,\,\,\, \text{for all}\,\, \, i,j=1,2,3,4.
$$
Note that we can write
$$
\Phi(\xi_1,\xi_2,\xi_3,\xi_4)= \xi|\xi|^{\alpha} P_{\eta_3}(\eta_1,\eta_2)
$$
where $P_{\eta_3}(\eta_1,\eta_2)= 1 - \sum_{j=1}^4 \eta_j|\eta_j|^{\alpha}$, $\eta_j= \frac{\xi_j}{\xi}$ and $\eta_4= 1 -\eta_1-\eta_2-\eta_3$. Since $\eta_3$ is fixed we have 
$$
\nabla P_{\eta_3}(\eta_1,\eta_2) = (\alpha+1
)\left(|\eta_4|^{\alpha} -|\eta_1|^{\alpha},|\eta_4|^{\alpha} -|\eta_2|^{\alpha}\right)
$$
and
$$
\nabla^2P_{\eta_3}(\eta_1,\eta_2)= -\alpha(\alpha+1)\left[
    \begin{tabular}{cc}
        $\frac{\eta_4}{{|\eta_4|^{2-\alpha}}}+ \frac{\eta_1}{{|\eta_1|^{2-\alpha}}}$ &$\frac{\eta_4}{{|\eta_4|^{2-\alpha}}}$ \\
        $\frac{\eta_4}{{|\eta_4|^{2-\alpha}}}$& $\frac{\eta_4}{{|\eta_4|^{2-\alpha}}}+ \frac{\eta_2}{{|\eta_2|^{2-\alpha}}}$ \\
    \end{tabular}
\right].
$$
Observe that, in this case, we have $|\eta_j \pm 1| \ll 1$ for all $j=1,2,3,4$. Additionally, it holds that
$$
\nabla P_{\pm 1}(\pm 1, \pm 1)= \vec{0},
$$
but 
$$
\det \left(\nabla^2 P_{\pm 1}\right)(\pm 1, \pm 1)\neq 0.
$$

Therefore, by Morse's Lemma (see, for example, \cite{MR2304165}, Lemma C.6.1), for every $\eta_3 \simeq \pm 1$, there exists a unique critical point $\mu(\eta_3):= (\mu_1(\eta_3),\mu_2(\eta_3))$ and a change of variables $(\eta_1,\eta_2) \mapsto (\zeta_1,\zeta_2)$ such that
$$
P_{\eta_3}(\eta_1,\eta_2,\eta_3,\eta_4) = P_{\eta_3}(\mu(\eta_3)) \pm \zeta_1^2 \pm \zeta_2^2.
$$
Then, applying some changes of variables
\begin{align*}
I_1 &\lesssim  \int  \frac{|\xi| \langle \xi \rangle^{s+1/2}}{\langle \xi_1 \rangle^{s+1/2}\langle \xi_2 \rangle^{s+1/2}\langle \xi_3 \rangle^{s}\langle \xi_4 \rangle^{s}} \1_{|\Phi -\beta|< M}\,  d\xi_1 d\xi_2\\
 &\lesssim \int  |\xi|^ {\frac{1}{2} - 3s} \1_{|\Phi -\beta|< M}\,  d\xi_1 d\xi_2\\
 &\lesssim \int  |\xi|^ {\frac{5}{2} - 3s} \1_{|\xi|\xi|^{\alpha}P_{\eta_3}(\eta_1,\eta_2) -\beta|< M}\,  d\eta_1 d\eta_2\\
 &\lesssim \int  |\xi|^ {\frac{5}{2} - 3s} \1_{|P_{\eta_3}(\mu(\eta_3)) \pm \zeta_1^2 \pm \zeta_2^2 -\beta|< \left(\frac{M}{\xi|\xi|^{\alpha}}\right)}\,  d\zeta_1 d\zeta_2.
\end{align*}
To proceed, we need the following auxiliary lemma.
\begin{lemma}[\hspace{-0.02cm}\cite{restrictionmethod}, Lemma 5] \label{lema 5}

For $N,M>0$ and $\beta \in \R$ fixed
$$
\int_{|p|,|q|<N} \1_{|p^2 \pm q^2 - \beta|<M} \, dp dq \lesssim M^{1^-}N^{0^+}.
$$
\end{lemma}
\qed

Therefore, by using Lemma \ref{lema 5} and  $s \geq s_3^+$ we obtain
\begin{align*}
\sup_{\xi,\xi_3, \beta} I_1(\xi, \xi_3) &\lesssim  \sup_{\xi}  |\xi|^ {\frac{5}{2} - 3s}\left(\frac{M}{\xi|\xi|^{\alpha}}\right)^{1^-}\\
&\lesssim  \sup_{\xi} |\xi|^ { \left(\frac{3}{2}- \alpha\right)^+ - 3s} M^{1^-}\\
&\lesssim  M^{1^-}.
\end{align*}
The estimates for 
$$
I_2(\xi_1,\xi_2,\xi_4) :=  \int \M_2 \1_{|\Phi -\beta|< M}\,  d\xi_1 d\xi_2 =  \int  \frac{|\xi| \langle \xi \rangle^{s-1/2}}{\langle \xi_1 \rangle^{s-1/2}\langle \xi_2 \rangle^{s-1/2}\langle \xi_3 \rangle^{s}\langle \xi_4 \rangle^{s}} \1_{|\Phi -\beta|< M}\,  d\xi,
$$
follow a similar argument.

\textbf{Case 2.} $|\xi| \simeq |\xi_1| \simeq |\xi_2| \simeq |\xi_3| \gg |\xi_4|$: Now, consider $A^\prime=\{0,4\}$, $\xi_3= \xi- \xi_1-\xi_2-\xi_4$ and, since $|\xi|>1$, we define
$$
\mathcal{N}_1(\xi_1,\xi_2,\xi_3,\xi_4) := \frac{ \langle \xi \rangle^{s+3/2}}{\langle \xi_1 \rangle^{s}\langle \xi_2 \rangle^{s}\langle \xi_3 \rangle^{s}\langle \xi_4 \rangle^{s}} \lesssim |\xi_3|^{3/2-3s}
$$
and
$$
\mathcal{N}_2(\xi_1,\xi_2,\xi_3,\xi_4) := \frac{\langle \xi \rangle^{s+ 1/2}}{\langle \xi_1 \rangle^{s}\langle \xi_2 \rangle^{s}\langle \xi_3 \rangle^{s}\langle \xi_4 \rangle^{s}} \lesssim |\xi|^{1/2-3s}.
$$
Thus, it is sufficient to estimate
\begin{equation}\label{I1}
\sup_{\xi_1,\xi_2,\xi_3, \beta} J_1(\xi_1,\xi_2,\xi_3) := \sup_{\xi_1,\xi_2,\xi_3, \beta} \int |\xi_3|^{3/2-3s} \1_{|\Phi -\beta|< M}\,  d\xi_4
\end{equation}
and
\begin{equation}\label{I2}
\sup_{\xi,\xi_4, \beta} J_2(\xi,\xi_4):= \sup_{\xi,\xi_4, \beta} \int |\xi|^{1/2-3s} \1_{|\Phi -\beta|< M}\,  d\xi_1 d\xi_2.
\end{equation}
Begin with \eqref{I1}. Note that, since $|\partial_{\xi_4} \Phi| \gtrsim |\xi_3|^{\alpha}$, applying the Hölder inequality and  performing the change of variables $\xi_4 \mapsto \Phi$
\begin{align*}
J_1=  \int |\xi_3|^{3/2-3s} \1_{|\Phi -\beta|< M}\,  d\xi_4 &\lesssim \left( \int |\xi_3|^{3/2^+-3s} \1_{|\Phi -\beta|< M}\,  d\xi_4 \right)^{1^-}\\
& \lesssim  \left( \int |\xi_3|^{3/2^+-3s} \1_{|\Phi -\beta|< M}\,  \frac{d \Phi}{|\partial_{\xi_4}\Phi|}\right)^{1^-} \\
& \lesssim  \left( \int |\xi_3|^{\left(3/2- \alpha\right)^+-3s} \1_{|\Phi -\beta|< M}\,  {d \Phi}\right)^{1^-}.
\end{align*}
Then, since $s> s_3^+$
\begin{align*}
\sup_{\xi_1,\xi_2,\xi_3, \beta} J_1(\xi_1,\xi_2,\xi_3) &\lesssim  \left( \int  \1_{|\Phi -\beta|< M}\,  {d \Phi}\right)^{1^-}\\
&\lesssim M^{1^-}.
\end{align*}
For \eqref{I2}, we will proceed similarly as in \textbf{Case 1.2}, i.e., write
$$
\Phi(\xi_1,\xi_2,\xi_3,\xi_4)= \xi|\xi|^{\alpha} Q_{\eta_4}(\eta_1,\eta_2),
$$
with $Q_{\eta_4}(\eta_1,\eta_2)= 1 - \sum_{j=1}^4 \eta_j|\eta_j|^{\alpha}$, $\eta_j= \frac{\xi_j}{\xi}$ and $\eta_3= 1 -\eta_1-\eta_2-\eta_4$. Since $\eta_4$ is fixed, we have
$$
\nabla^2Q_{\eta_4}(\eta_1,\eta_2)\simeq \left[
   \begin{tabular}{cc}
        $\frac{\eta_3}{{|\eta_3|^{2-\alpha}}}+ \frac{\eta_1}{{|\eta_1|^{2-\alpha}}}$ &$\frac{\eta_3}{{|\eta_3|^{2-\alpha}}}$ \\
        $\frac{\eta_3}{{|\eta_3|^{2-\alpha}}}$& $\frac{\eta_3}{{|\eta_3|^{2-\alpha}}}+ \frac{\eta_2}{{|\eta_2|^{2-\alpha}}}$ \\
    \end{tabular}
\right].
$$
Therefore, since $\eta_1 \simeq \eta_2 \simeq 1$, $\eta_3 \simeq -1$ we have
$$
\det \left(\nabla^2 Q_{\eta_4}\right) \simeq -1.
$$
Thus, is sufficiently  consider two cases:

\textbf{Case 2.1.} $|\nabla Q_{\eta_4}| \gtrsim 1$ \textit{(Non-Stationary region)}: In this scenario, being away from a stationary point, we can, without loss of generality, assume $|\partial_{\eta_1} Q_{\eta_4}| \gtrsim 1$. Thus, after performing some variable changes
\begin{align*}
J_2 &= \int |\xi|^{1/2-3s} \1_{|\Phi -\beta|< M}\,  d\xi_1 d\xi_2\\
&= \int_{\eta_2 \simeq 1} |\xi|^{5/2-3s} \1_{|\xi|\xi|^{\alpha}Q_{\eta_4} -\beta|< M}\,  d\eta_1 d\eta_2\\
&\lesssim \left(\int_{\eta_2 \simeq 1} |\xi|^{5/2^+-3s} \1_{|\xi|\xi|^{\alpha}Q_{\eta_4} -\beta|< M}\,  d\eta_1 d\eta_2 \right)^{1^-}\\
&\lesssim \left(\int_{\eta_2 \simeq 1} |\xi|^{5/2^+-3s} \1_{|\xi|\xi|^{\alpha}Q_{\eta_4} -\beta|< M}\,  \frac{dQ_{\eta_4}}{|\partial_{\eta_1} Q_{\eta_4}| } d\eta_2 \right)^{1^-}\\
&\lesssim \left(\int_{\eta_2 \simeq 1} |\xi|^{5/2^+-3s} \1_{|\xi|\xi|^{\alpha}Q_{\eta_4} -\beta|< M}\,  {dQ_{\eta_4}} d\eta_2 \right)^{1^-}.
\end{align*}
Then, by applying Lemma \ref{lema 5} once more
\begin{align*}
\sup_{\xi,\xi_4, \beta} J_2(\xi,\xi_4) &\lesssim  \sup_{\xi}  |\xi|^ {\frac{5}{2} - 3s}\left(\frac{M}{\xi|\xi|^{\alpha}}\right)^{1^-}\\
&\lesssim  \sup_{\xi} |\xi|^ { \left(\frac{3}{2}- \alpha\right)^+ - 3s} M^{1^-}\\
&\lesssim  M^{1^-}.
\end{align*}

\textbf{Case 2.2.} $|\nabla Q_{\eta_4}| \ll 1$ \textit{(Stationary region)}: Applying Morse's Lemma, there exists a change of variable $(\eta_1,\eta_2) \mapsto (\zeta_1,\zeta_2)$ such that
$$
Q_{\eta_4}(\eta_1,\eta_2) = Q_{\eta_4}(\zeta) \pm \zeta_1^2 \pm \zeta_2^2
$$
where $\zeta \in \R^2$ is a non-degenerate critical point of $Q$. Then
\begin{align*}
J_2 &= \int |\xi|^{1/2-3s} \1_{|\Phi -\beta|< M}\,  d\xi_1 d\xi_2\\
&= \int_{\eta_2 \simeq 1} |\xi|^{5/2-3s} \1_{|\xi|\xi|^{\alpha}Q_{\eta_4} -\beta|< M}\,  d\eta_1 d\eta_2\\
&\lesssim \left(\int_{\eta_2 \simeq 1} |\xi|^{5/2^+-3s} \1_{|\xi|\xi|^{\alpha}Q_{\eta_4} -\beta|< M}\,  d\eta_1 d\eta_2 \right)^{1^-}\\
&\lesssim \left(\int |\xi|^{5/2^+-3s} \1_{|(Q_{\eta_4}(\zeta) \pm \zeta_1^2 \pm \zeta_2^2) -\beta|< \frac{M}{\xi|\xi|^{\alpha}}}\,  d\zeta_1 d\zeta_2 \right)^{1^-}.
\end{align*}
Applying the Lemma \ref{lema 5} once again and recalling that $s \geq s_3^+$
\begin{align*}
\sup_{\xi,\xi_4, \beta} J_2(\xi,\xi_4) \lesssim  M^{1^-}.
\end{align*}
By these estimates and applying Lemma \ref{Lema 2}, the estimate \eqref{estimativa multilinear bourgain} follows in this region.

\qed

Note that, if we consider $s^\prime = s + \varepsilon $ with $\varepsilon<\min \left\{3s+\alpha - \frac{3}{2},s+1/2,\alpha-1\right\}$ by following the same approach used in the proof of Theorem \ref{Teorema em XSB}, we obtain
\begin{equation}\label{estimativa multilinear com epsilon}
   \| \partial_x(u_1u_2u_3u_4)\|_{X^{s+\varepsilon,-\frac{1}{2}^-}_\alpha} \lesssim \|u_1\|_{X^{s,\frac{1}{2}^+}_\alpha}\|u_2\|_{X^{s,\frac{1}{2}^+}_\alpha}\|u_3\|_{X^{s,\frac{1}{2}^+}_\alpha}\|u_4\|_{X^{s,\frac{1}{2}^+}_\alpha}, 
\end{equation}

for any regularity $s>s_3$. Using the estimate \eqref{estimativa multilinear com epsilon}, Theorem \ref{efeito regularizante nao linear} follows. \qed

\subsection{Global theory for negative indices}
This subsection is dedicated to proving Theorem \ref{BOA COLOCAÇAO GLOBAL I METHOD}. Therefore, from this point forward, we will consider $\alpha \in (\frac{3}{2},2]$. In this section is important to consider the restriction norm
$$
\|f\|_{X^{s,b}_\alpha(T)}:= \left\{ \|g\|_{X^{s,b}_\alpha}:g_{|_{[0,T]\times \R} = f } \right\}.
$$
Henceforth, to the sake of clarity, we will simplify the notation by writing $\|f\|_{X^{s,b}_\alpha(T)}$ as $\|f\|_{X^{s,b}_\alpha}$.

To prove Theorem \ref{BOA COLOCAÇAO GLOBAL I METHOD}, we  follow the argument in \cite{MR1824796} and \cite{MR2372424}. For that, given $N \gg 1$ and $s<0$, define the multiplier
\begin{equation*}
	m_N(\xi):= \left\{
	\begin{array}{l}
	1, \quad |\xi|<N\\
	\left(N^{-1}|\xi|\right)^s, \quad |\xi|>N,
	\end{array}
	\right.
\end{equation*}
and the operator
$$
I_N f(x):= \left( m_N(\xi) \widehat{f}(\xi)\right)^\vee(x).
$$
Note that, for fixed $N \gg 1$, we have
$$
\|I_Nf\|_{L^2_x} \simeq \|f\|_{H^s_x}.
$$
In other words, $I_N$ is a "almost-isomorphism" between $L^2_x(\R)$ and $H^s_x(\R)$. Moreover, $\|I_N \cdot\|_{L^2_x}$ defines an equivalent norm on $H^s_x(\R)$.

Now, let $u_0 \in H^s_x(\R)$ with $s_3 < s \leq 0$. Therefore exist a $T=T(\|u_0\|_{H^s_x})>0$ and a unique solution $u:[0,T] \times \R \rightarrow \R$ to \eqref{DGBOl2}. Note that, by integration by parts and the Cauchy-Schwarz inequality we obtain, formally,  for any $N \gg 1$
\begin{align*}
\|I_N u(T)\|_{L^2_x}^2 - &\|I_N u_0\|_{L^2_x}^2 \\
&= \int_0^T \frac{d}{dt} \int_\R I_N u(\tau, x) I_N u (\tau,x)\, dx \, d\tau \\
&=2\int_0^T \int_\R  I_N\left[u_t(\tau, x) \right]I_N u (\tau,x)\, dx \, d\tau \\
&=2\int_0^T \int_\R  I_N\left[-D_x^{\alpha} u_x -  u^3u_x \right](\tau,x) I_N u (\tau,x)\, dx \, d\tau\\
&=-\frac{1}{2}\int_0^T \int_\R  I_N\left[\partial_x(u^4)(\tau,x) \right]I_N u (\tau,x)\, dx \, d\tau\\
&=\frac{1}{2}\int_0^T \int_\R  \partial_x \left[ (I_Nu)^4 - I_N(u^4) \right](\tau,x)I_N u (\tau,x)\, dx \, d\tau\\
&\lesssim \|\partial_x \left[ I_N(u^4) - (I_Nu)^4 \right]\|_{X_\alpha^{0,-\frac{1}{2}^-}} \|I_Nu\|_{X_\alpha^{0,\frac{1}{2}^+}}.
\end{align*}
In the next lemma, we utilize the same approach as in \cite{MR2372424}.

\begin{lemma}
The estimate 
\begin{equation}\label{estimativa multi I_N}
\|\partial_x [ I_N(u_1u_2u_3u_4) - \prod_{j=1}^4 I_Nu_j]\|_{X_\alpha^{0,-\frac{1}{2}^-}} \lesssim N^{- \left(\alpha- \frac{3}{2}\right)^-} \prod_{j=1}^4\|I_Nu_j \|_{X_\alpha^{0,\frac{1}{2}^+}},  
\end{equation}
holds.
\end{lemma}

\dem Denote by $\xi_j$  the frequencies corresponding to the function $u_j$, with $1\leq j \leq 4$. If all $ |\xi_j| \leq N $, then either $ |\xi| \leq N $, resulting in no contribution. Thus we can assume, without loss of generality,  there are $k$ large frequencies for some $1 \leq k \leq 4$. By symmetry, we can assume $|\xi_j|\geq {N}$ with $j=1,...,k$ and $|\xi_j| \leq N$ for $j=k+1,...,4$.  By applying arguments similar to those used to derive \eqref{estimativa multilinear com epsilon} (see \cite{correia2024improvedglobalwellposednessquartic}, Lemma 10), it is possible to obtain
\begin{align*}
\|\partial_x \left[I_Nu_1I_Nu_2I_Nu_3I_Nu_4\right]\|_{X^{0,-\frac{1}{2}^-}_\alpha} & \lesssim \|I_Nu_1\|_{X^{-\frac{\varepsilon}{k},\frac{1}{2}^+}_\alpha}\cdots\|I_Nu_k\|_{X^{-\frac{\varepsilon}{k},\frac{1}{2}^+}_\alpha} \\
& \quad \times \|I_Nu_{k+1}\|_{X^{0,\frac{1}{2}^+}_\alpha}\cdots\|I_Nu_4\|_{X^{0,\frac{1}{2}^+}_\alpha}\\
&\lesssim N^{- k\frac{\varepsilon}{k}}\|I_Nu_1\|_{X^{0,\frac{1}{2}^+}_\alpha}\cdots\|I_Nu_k\|_{X^{0,\frac{1}{2}^+}_\alpha} \\
&\quad  \times  \|I_Nu_{k+1}\|_{X^{0,\frac{1}{2}^+}_\alpha}\cdots\|I_Nu_4\|_{X^{0,\frac{1}{2}^+}_\alpha}\\
& \simeq N^{- \varepsilon} \prod_{j=1}^4\|I_Nu_j \|_{X^{0,\frac{1}{2}^+}_\alpha},
\end{align*}
where $\varepsilon=\left(\alpha- \frac{3}{2}\right)^-$. Similarly, we obtains
\begin{align*}
\|\partial_x  I_N(u_1u_2u_3&u_4)\|_{X^{0,-\frac{1}{2}^-}_\alpha}\\
&\lesssim \|u_1\|_{X^{-\frac{\varepsilon}{k},\frac{1}{2}^+}_\alpha}\cdots\|u_k\|_{X^{-\frac{\varepsilon}{k},\frac{1}{2}^+}_\alpha}\|u_{k+1}\|_{X^{0,\frac{1}{2}^+}_\alpha}\cdots\|u_4\|_{X^{0,\frac{1}{2}^+}_\alpha}\\
& \simeq \|u_1\|_{X^{-\frac{\varepsilon}{k},\frac{1}{2}^+}_\alpha}\cdots\|u_k\|_{X^{-\frac{\varepsilon}{k},\frac{1}{2}^+}_\alpha}\|I_Nu_{k+1}\|_{X^{0,\frac{1}{2}^+}_\alpha}\cdots\|I_Nu_4\|_{X^{0,\frac{1}{2}^+}_\alpha}.
\end{align*}
Since, for any $r_1 \leq r_2$ it is valid that for $j=1,...,k$
\begin{align*}
\|u_j\|_{X^{r_1,\frac{1}{2}^+}_\alpha} &\lesssim N^{r_1-r_2} \|u_j\|_{X^{r_2,\frac{1}{2}^+}_\alpha}\\
& \lesssim N^{r_1} \|I_Nu_j\|_{X^{0,\frac{1}{2}^+}_\alpha},
\end{align*}
we can conclude
\begin{align*}
\|\partial_x \left[ I_N(u_1u_2u_3u_4)\right]\|_{X^{0,-\frac{1}{2}^-}_\alpha} &\lesssim N^{- \varepsilon} \prod_{j=1}^4\|I_Nu_j \|_{X^{0,\frac{1}{2}^+}_\alpha}.
\end{align*}
Which finish the proof of the lemma.\qed

By applying \eqref{estimativa multi I_N} to the previous estimates, we obtain that
\begin{align}
\nonumber \|I_N u(T)\|_{L^2_x}^2 - \|I_N u_0\|_{L^2_x}^2 &\lesssim \|\partial_x \left[ I_N(u^4) - (I_Nu)^4 \right]\|_{X^{0,-\frac{1}{2}^-}} \|I_Nu\|_{X^{0,\frac{1}{2}^+}}\\
&\lesssim N^{- \left(\alpha- \frac{3}{2}\right)^-} \|I_Nu \|_{X^{0,\frac{1}{2}^+}_\alpha}^5.\label{I_N}
\end{align}
To advance our argument, we need the following variation of the local well-posedness result.
\begin{theorem}[Variant of local well-posedness]\label{LWP para IN}
The problem \eqref{DGBOl2} is locally well-posed in $I^{-1}_NL^2_x:= \left(H^s_x(\R),\|I_N \cdot\|_{L^2_x} \right)$ for $s>s_3$ with time of existence  $T>0$. Furthermore, we have 
$$
T \gtrsim \|I_Nu_0\|_{L^2_x}^{-\theta}
$$
for some $\theta>0$ and
\begin{equation}\label{I_N (2)}
\|I_N u \|_{X^{0, \frac{1}{2}^+}_\alpha} \lesssim \|I_Nu_0\|_{L^2_x}.
\end{equation}
\end{theorem} 

The proof of this result follows a similar argument to the one presented earlier and will be omitted here. \qed

The estimates \eqref{I_N} and  \eqref{I_N (2)} ensure that
\begin{equation}\label{Lei de quase conservação}
\|I_N u(T)\|_{L^2_x}^2 - \|I_N u_0\|_{L^2_x}^2 \lesssim N^{- \left(\alpha- \frac{3}{2}\right)^-} \|I_Nu_0 \|_{L^2_x}^5.
\end{equation}

Now, let $u$ be a solution of \eqref{DGBOl2} in time interval $[0,T]$. Then
$$
u_\lambda (t,x)= \lambda^{-\frac{\alpha}{3}}u\left(\frac{t}{\lambda^{\alpha+1}}, \frac{x}{\lambda} \right),
$$
is a solution with time interval $[0, \lambda^{\alpha+1} T]$ and initial data $u^\lambda_0 (x)= \lambda^{-\frac{\alpha}{3}}u_0 (\lambda^{-1} x)$. A simple calculation shows that
$$
\|I_N u_0^{\lambda}\|_{L^2_x} \lesssim N^{-s} \lambda^{\frac{1}{2}-\frac{\alpha}{3}-s} \|I_N u_0\|_{L^2_x}.
$$
Choose $\lambda=\lambda(N)$ so that
$$
\lambda \sim N^{\frac{6s}{3-2 \alpha-6s}}.
$$
Thus obtaining that
$$
\|I_N u_0^{\lambda}\|_{L^2_x} \lesssim \|u_0\|_{H^s_x}:= \delta_0.
$$
Our goal is to construct the solution of \eqref{DGBOl2} over the time interval $[0, \lambda^{\alpha+1} T]$. Taking $\delta_0 \ll 1$ and iterating Lemma \ref{LWP para IN}, for any $M>0$, we obtain a solution $u_\lambda$ defined in $[0,M]$ with $u_\lambda(0)= u_0^\lambda$, as long as
\begin{equation}\label{condição de iteração IN}
  \|I_N u_\lambda(k)\|_{L^2_x} \lesssim \delta_0,  
\end{equation}
for $1 \leq k <M$.  Using \eqref{Lei de quase conservação}, we have
\begin{align*}
\|I_N u_\lambda(k)\|_{L^2_x}^2 &\lesssim \delta_0^2 + kN^{- \left(\alpha- \frac{3}{2}\right)^-}\delta_0^5\\
& \lesssim \delta_0 \left(1+kN^{- \left(\alpha- \frac{3}{2}\right)^-}\right).
\end{align*}
Taking $s> -\frac{(2\alpha -3)^2}{24\alpha - 6}$ and $N \gg 1$ large enough, we have
$$
N^{\left(\alpha- \frac{3}{2}\right)^+} \gtrsim  N^{\frac{6s(\alpha+1)}{3-2 \alpha-6s}}T \sim \lambda^{\alpha+1} T .
$$
Thus, the condition \eqref{condição de iteração IN} is satisfied for $ j \leq \lambda^{\alpha+1}T $. Hence, the solution $u_\lambda$ is defined in time interval $[0,\lambda^{\alpha+1}T]$. Therefore,
$$
u(t,x) := \lambda^{\frac{\alpha}{3}} u\left(\lambda^{\alpha+1} t, \lambda x\right)
$$
is a solution of \eqref{DGBOl2} on $[0, T]$ for any arbitrary $ T > 0 $. Additionally, since
$$
\lambda \sim N^{\frac{6s}{3-2 \alpha-6s}} \quad \text{and} \quad\lambda^{\alpha+1} T\lesssim N^{\left(\alpha- \frac{3}{2}\right)^+},
$$
after performing some calculations, we find that
$$
\lambda\lesssim T^{\frac{6\kappa(s,\alpha)}{2\alpha-3}}
$$
with 
$$
\kappa(s,\alpha):=\frac{2s}{(2\alpha-3)^2(3-2\alpha-6s)-12(2\alpha-3)(\alpha+1)s}.
$$
Thus, for $0 \leq t \leq T$, by \eqref{condição de iteração IN}, we obtain
\begin{align*}
\|u(t)\|_{H^s_x} &\lesssim \lambda^{\frac{\alpha}{3}- \frac{1}{2}} \|I_N  u_\lambda(\lambda^{\alpha+1}t)\|_{L^2_x}\\
& \lesssim T^{\kappa(s,\alpha)}  \|u_0\|_{L^2_x}.
\end{align*}
Consequently, the proof of Theorem \ref{BOA COLOCAÇAO GLOBAL I METHOD} is complete. 

\qed

\textbf{Acknowledgments:} The author L. Campos was partially supported by CNPq grant 07733/2023-8 and the CAPES/Cofecub grant 88887.879175/2023-00. F. Linares was partially supported by CNPq grant 310329/2023-0 and FAPERJ grant E-26/200.465/2023. The author T.S.R. Santos was partially supported by CNPq grant 140993/2021-5 and is grateful to IST, Universidade de Lisboa, where part of this work was carried out. The authors would like to thank Argenis Mendez for his valuable feedback on an earlier version of this manuscript.





\bibliographystyle{abbrv}

{{
  \bigskip
  \footnotesize
\vspace{1cm}
L. Campos, \textsc{Departamento de Matemática, UFMG, Belo Horizonte, MG, Brazil.} \textit{E-mail address:} \texttt{luccas@ufmg.br}.
  
  \vspace{3mm}

F. Linares, \textsc{IMPA, Instituto de Matemática Pura e Aplicada, Rio de Janeiro, RJ, 22460-320, Brazil.} \textit{E-mail address:} \texttt{linares@impa.br}.
    
  \vspace{3mm}

T. S. R. Santos, \textsc{IMPA, Instituto de Matemática Pura e Aplicada, Rio de Janeiro, RJ, 22460-320, Brazil.} \textit{E-mail address:} \texttt{thyago.souza@impa.br}.
}}

 \end{document}